\documentclass{article}
\usepackage{graphicx}

\usepackage[fleqn]{amsmath}

\usepackage{amsthm}
\usepackage{amssymb}
\usepackage{amsbsy}
\usepackage{amsmath}
\usepackage{amsfonts}
\usepackage{mathrsfs}
\usepackage{amsmath}
\usepackage[all]{xy}
\usepackage{amstext}
\usepackage{amscd}
\usepackage[dvips]{epsfig,color}
\usepackage{psfrag}
\usepackage{enumerate}
\usepackage{flafter}
\allowdisplaybreaks

\theoremstyle{plain}
\newtheorem{thm}{Theorem}[section]
\newtheorem{prop}[thm]{Proposition}

\newtheorem{lem}[thm]{Lemma}
\theoremstyle{definition}
\newtheorem{exa}[thm]{Example}

\newtheorem{rem}[thm]{Remark}
\newtheorem{defn}[thm]{Definition}
\newtheorem{prob}[thm]{Problem}

\def\dim{\mathop{\mathrm{dim}}\nolimits}

\def\det{\mathop{\mathrm{det}}\nolimits}
\def\Im{\mathop{\mathrm{Im}}\nolimits}
\def\Qer{\mathop{\mathrm{Ker}}\nolimits}
\def\Coker{\mathop{\mathrm{Coker}}\nolimits}
\def\Hom{\mathop{\mathrm{Hom}}\nolimits}

\newcommand{\lra}{\longrightarrow}
\newcommand{\ra}{\rightarrow}
\newcommand{\N}{{\Bbb N}}
\newcommand{\Q}{{\Bbb Q}}

\newcommand{\Z}{{\Bbb Z}}

\newcommand{\Aut}{{\rm Aut}}
\newcommand{\Out}{{\rm Out}}
\newcommand{\Sp}{{\rm Sp}}
\newcommand{\GL}{{\rm GL}}
\newcommand{\SpAut}{{\rm SpAut}}
\newcommand{\SpOut}{{\rm SpOut}}
\newcommand{\lc}{{\mathrm{l}\textrm{-}\mathrm{coef}}}
\newcommand\bigzerou{\smash{\lower.3ex\hbox{\Huge 0}}}

\newcommand{\id}{{\mathrm{id}}}

\begin{document}
\large
\begin{center}
{\bf\Large Meta-nilpotent knot invariants and}
\end{center}
\begin{center}
{\bf\Large symplectic automorphism groups of free nilpotent groups}
\end{center}
\begin{center}{Takefumi Nosaka\footnote{
E-mail address: {\tt nosaka@math.titech.ac.jp}
}}\end{center}
\begin{abstract}\baselineskip=12pt \noindent
We develop the nilpotent $p$-localization of knot groups in terms of the (symplectic) automorphism groups of free nilpotent groups.
We show that any map from the set of conjugacy classes of the outer automorphism groups defines a knot invariant.
Furthermore, we investigate these automorphism groups and compute the resulting knot invariants.
\end{abstract}
\begin{center}
\normalsize
\baselineskip=11pt
{\bf Keywords} \\
\ \ \ Meta-nilpotent quotient, knot, symplectic representation, mapping class group, Milnor pairing \ \ \
\end{center}
\begin{center}
\normalsize
\baselineskip=11pt
{\bf Subject Codes } \\
\ \ \ 57K10,57K18,20J06, 20F18,20F28 \ \ \
\end{center}

\large
\baselineskip=16pt
\section{Introduction}
In this paper, we study the meta-nilpotence of a knot $K$ in an integral homology 3-sphere\footnote{{\it An integral homology 3-sphere} is a closed 3-manifold having integral homology of the 3-sphere. } $M$ and suggest an approach to such knots regarding certain automorphism groups.

First, we will establish the terminology regarding nilpotent groups. For a group $G$ and subgroups $K,H \subset G$, let $[K,H]$ be the commutator subgroup generated by elements $hkh^{-1}k^{-1}$ with $h \in H, k\in K$. Define $G_1$ to be $[G,G]$, and inductively $G_{k+1} $ to be $[G,G_k]$. Namely, $G \supset G_1 \supset G_2 \supset \cdots $ is a lower central series. For any prime $p$ (possibly $p=0$), we can define the localization of the nilpotent quotient $G/G_k$ and denote it by $G/G_k\otimes \Z_{(p)} $ {where $\Z_{(p)}$ is the localization of $\Z$ at $p$, that is, $\Z_{(p)}=\{ r/s \in \Q \mid r,s \in \Z,\ p \textrm{ does not divide }s \}$; see Section \ref{Ss45} for details. If the abelianization of $G$ is $\Z$, we have a semi-direct product $[G,G] \rtimes \Z$; the restricted action of $\Z$ on $[G,G]_k$ ensures a semi-direct product $ [G,G]/[G,G]_k \rtimes \Z$, which we call {\it the ($k$-th) meta-nilpotent quotient (of $G$).}

In Section \ref{Ss45}, we develop a nilpotent approach to studying knots in the spirit of fibered spaces with a surface fiber. The correspondence between fibered knots and the associated monodromy is a complete invariant of fibered knots; however, some knots are fibered, but others are not (we discuss some properties of fiberedness in Section \ref{Ss3}). As a generalization suitable for every knot, we will show (Theorem \ref{bb1279}) that, for an appropriate prime $p$, the localization of the meta-nilpotent quotient of any knot group $\pi_K:= \pi_1(M \setminus K) $ is isomorphic to that of a free group $F$, that is,
$$ ([\pi_K,\pi_K]/([\pi_K,\pi_K])_k \otimes \Z_{(p)}) \rtimes \Z \cong (F/F_k \otimes \Z_{(p)}) \rtimes \Z,$$
where the rank of $F$ is equal to the degree of the Alexander polynomial $\Delta_K$.
If $p=0$, the isomorphism is mentioned and studied from the viewpoint of ``Fox pairings"; see \cite{Tur2}

Moreover, we completely classify the choices of a semi-direct product, up to conjugacy, as follows. 
By considering the projection of outer automorphism groups,
$$q_k: \Out(F/F_{k } \otimes \Z_{(p)}) \ra \Out(F/F_{1} \otimes \Z_{(p)}) =\GL(\mathrm{rank}F; \Z_{(p)}),$$
we take the subgroup, $ \SpOut(F/F_{k } \otimes \Z_{(p)}) $, of the preimage of the symplectic group $\Sp(\mathrm{rank}F; \Z_{(p)})$. Then, we show (Theorem \ref{bb1279}) that, when the action of $\Z$ in the semi-direct product is regarded as an automorphism in $\Out(F/F_{k } \otimes \Z_{(p)}) $, it lies in $ \SpOut(F/F_{k } \otimes \Z_{(p)})$ by Milnor duality \cite{Mil}; Theorem \ref{ii98788} shows that the correspondence between a knot and the monodromy of a semi-direct product defines a map
\begin{equation}\label{5529}
\{ \textrm{a knot }K \textrm{ in an integral homology 3-sphere with deg}\Delta_K=\mathrm{rank}F \} \lra \frac{\SpOut(F/F_{k } \otimes \Z_{(p)})}{ \mathrm{conjugation}}.
\end{equation}
In Section \ref{93553}, we will see that the restriction on the set of fibered knots is {\it a priori} equal to the Dehn-Nielsen embedding. In other words, the construction of such monodromies is an extension of the correspondence between a fibered knot and the monodromy.
Section \ref{93553} also discusses the image and surjectivity of the map \eqref{5529}, as a comparison with the Dehn-Nielsen embedding and homology cobordism;
see Propositions \ref{llk1} and \ref{llk1456}. 


Meanwhile, the map \eqref{5529} implies that any map from the conjugacy classes of $\SpOut(F/F_{k } \otimes \Z_{(p)}) $ is a knot invariant. To obtain such maps, we invoke the works on the Johnson homomorphisms of the automorphism group $\Aut(F/F_{k}) $ and of the mapping class group, $\mathcal{M}_{g,1}$, of an orientable surface with one boundary. Roughly speaking, the Johnson homomorphism is a nilpotent study of $\Aut(F/F_{k})$ and of some homomorphism $\rho_k: \mathcal{M}_{g,1} \ra \Aut(F/F_{k})$, together with the help of $\Sp$-representations. Inspired by the works of Morita \cite{Mo,Mo2,Mo3}, in Section \ref{9914} we investigate the group $\SpOut(F/F_{k} \otimes \Z_{(p)})$ in detail, while in Section \ref{99sec} we explicitly describe its structure for $k=2,3$ in terms of $\Sp$-representations. As a corollary, we define certain maps from the conjugacy classes of $\SpOut(F/F_{k} \otimes \Z_{(p)})$ for $k=1,2,3$; see Section \ref{Seinv}.

This paper is organized as follows. Section 2 reviews several properties of fibered knots. Sections 3 and 4 deal with the nilpotent localization of knot groups and establish uniqueness up to conjugacy. Section 5 compares our results with the Dehn-Nielsen embedding. Sections 6 and 7 examine the automorphism groups $\Aut(F/F_k \otimes \Z_{(p)})$ and $\Out(F/F_k \otimes \Z_{(p)})$. Finally, Section \ref{Seinv} proposes knot invariants derived from the nilpotent localization.

\

\noindent
{\bf Conventional notation.} \ Every knot $K$ is understood to be smooth, oriented, and embedded in a rational homology 3-sphere $M$ as a circle, where $K$ is null-homologous in $H_1(M;\Z)$. We denote by $\pi_K$ the knot group $\pi_1(M \setminus K )$, and by $\Delta_K \in \Q[t^{\pm 1 }]$ the Alexander polynomial of $K$; see \cite{Lic} for the definition. By $p$ we mean a prime number (possibly $p=0$). For a polynomial $f(t)$, let $\lc f(t)$ be the leading coefficient.


Given groups $G$ and $H$ and a homomorphism $\phi: G \to \Aut(H)$, we can define the semidirect product $H \rtimes G$. When it is necessary to emphasize the role of $\phi$, we denote this by $H \rtimes_{\phi} G$. If $G = \Z$, such an $H \rtimes G$ is sometimes called a {\it mapping torus}.

Moreover, if there exists a surjective homomorphism $p: \Gamma \to \Z$, we can choose a homomorphism $\mathfrak{s}: \Z \to \Gamma$ as a section, which induces a group isomorphism
$\Gamma \cong \mathrm{Ker}(p) \rtimes_{\phi} \Z$, 
where $\phi(m)(h) =  \mathfrak{s} (m) h (\mathfrak{s} (m) )^{-1}$ for $m \in \Z, h \in \mathrm{Ker}(p)$. For example, for any knot $K \subset S^3$, the knot group $\Gamma= \pi_K$ is a semidirect product $[\pi_K,\pi_K] \rtimes \Z$.

\section{Fibered knots and monic Alexander polynomials}
\label{Ss3}
We will mainly focus on localized meta-nilpotent quotients of knot groups. Before discussing the quotient, we review fibered knots as a toy model. A knot $K \subset M$ is said to be {\it fibered} if $M \setminus K$ is the total space of a fiber bundle over $S^1$ whose fiber is a Seifert surface. It is worth noting that if $K$ is fibered, then $\pi_1(M \setminus K) \cong F \rtimes \Z$ for some free group $F$; in particular, the $k$-th meta-nilpotent quotient of $\pi_1(M \setminus K)$ is isomorphic to $F/F_k \rtimes \Z$.

Conversely, as is classically known, if $K$ is not fibered, no free group $F$ admits $\pi_1(M \setminus K) \cong F \rtimes \Z$. However, as indicated in \cite[\S 5]{Gor}, we can show that, under a certain condition, the meta-nilpotent quotient of $\pi_1(M \setminus K)$ is a mapping torus: more precisely,

\begin{prop}[{folklore}]\label{bb12}
Let $M$ be an integral homology 3-sphere. Let $\pi_K $ be the knot group $\pi_1(M \setminus K) $ of a knot $K \subset M $. Suppose that the leading coefficient of $\Delta_K$ is $\pm 1$, i.e., $\lc \Delta_K = \pm 1. $

Then, there exist a free group $F$ of rank $\mathrm{deg} \Delta_K $ and an isomorphism $\tau : F \ra F$ such that the $k$-th meta-nilpotent quotient of $\pi_1(M \setminus K)$ is isomorphic to $F/F_k \rtimes_{\tau} \Z$ for any $k \in \N$.
\end{prop}

\begin{proof}
Rapaport \cite{Ra} and Crowell \cite{Cro} prove that, since $\Delta_K $ is monic, $H_1([\pi_K, \pi_K]; \Z)$ is finitely and freely generated and of rank $\mathrm{deg} \Delta_K $, that is, $H_1([\pi_K, \pi_K]; \Z) \cong \Z^{\mathrm{deg} \Delta_K} $. For $m \geq 2$, the homology $H_m([\pi_K, \pi_K]; \Z)=0$ is widely known; see, e.g., \cite{Cro}. Therefore, it follows from \cite[Theorem 3.4]{Sta} that $[\pi_K, \pi_K]/([\pi_K, \pi_K])_k$ is isomorphic to $F/F_k$ for some free group $F$. Since an action of $\Z$ on $[\pi_K, \pi_K]/([\pi_K, \pi_K])_k$ can be identified with that on $F/F_k$, we have $[\pi_K, \pi_K]/([\pi_K, \pi_K])_k \rtimes \Z \cong F /F_k \rtimes_{\tau} \Z $ for some $\tau$, as required.
\end{proof}
We end this section with a remark on the assumption that $\Delta_K(t)$ is monic. As is known \cite{Cro,Ra}, if this is not the case, then $H_1([\pi_K, \pi_K]; \Z)$ is not finitely generated over $\Z$. Consequently, the $k$-th meta-nilpotent quotient can never be of the form $F/F_k \rtimes_{\tau} \Z$ for any free group $F$ of finite rank.

\section{Meta-nilpotent quotients of knot groups}
\label{Ss45}
While $\Delta_K(t) $ is required to be monic in Proposition \ref{bb12}, here, we give a similar theorem (Theorem \ref{bb1279}) that is suitable to every knot. For this, we will discuss meta-nilpotent quotients.

First, let us briefly review the $p$-localization of nilpotent groups. Here, we refer the reader to \cite{Rag}. For any group $G$ and $k \in \mathbb{N} \cup \{ \infty\} $, there exist a unique group $G/G_k \otimes \Z_{(p)}$ and a homomorphism $\kappa: G/G_k \ra G/G_k \otimes \Z_{(p)}$ satisfying that
\begin{enumerate}[(R1)]
\item For any $m \in \Z$, the quotient homomorphism 
\[\bar{ \kappa} :G_m/G_{m+1} \lra ( G/G_k \otimes \Z_{(p)})_m / ( G/G_k \otimes \Z_{(p)})_{m+1} \]
induced by $\kappa$ is equal to the localization at $p$ of the additive homomorphism $G_m/G_{m+1} \ra (G_m/G_{m+1})_{(p)}$ at $p$.
\item The induced map $\kappa_* : H_*(G/G_k ; \Z) \ra H_*(G/G_k \otimes \Z_{(p)} ;\Z )$ on the group homology with coefficients $\Z$ is the localization of $\Z$ at $p$.
\end{enumerate}

The purpose of this section is to show the following:
\begin{prop}\label{bb127}
Let $\tau: G \ra G$ be an automorphism. Then, there exists a group isomorphism $\tau\otimes \Z_{(p)} : G/G_k \otimes \Z_{(p)} \ra G/G_k \otimes \Z_{(p)} $ such that $(\tau\otimes \Z_{(p)}) \circ \kappa = \kappa \circ \tau$ and the quotient map
\[(\tau\otimes \Z_{(p)})_* : \frac{(G/G_k \otimes \Z_{(p)})_m }{( G/G_k \otimes \Z_{(p)})_{m+1}} \lra \frac{(G/G_k \otimes \Z_{(p)})_m }{( G/G_k \otimes \Z_{(p)})_{m+1}} \]
is the $p$-localization of the quotient $ \overline{\tau} : G_m/ G_{m+1} \ra G_m/ G_{m+1} $ for any $m \in \N.$

Furthermore, the existence of $\tau\otimes \Z_{(p)} $ is unique up to conjugacy.
\end{prop}

\begin{thm}\label{bb1279}
Let $K$ be a knot in an integral homology 3-sphere $M$, and let $F$ be the free group of rank $ \mathrm{deg} \Delta_K = 2g$. Suppose $p$ and $\lc \Delta_K$ are relatively prime, (possibly $p=0$). 

 For any $m \in \N $, there exist an automorphism $\tau : F/F_m \otimes \Z_{(p)} \ra F /F_m \otimes \Z_{(p)}$ and a $\Z$-equivariant homomorphism $\psi_m : \pi_K \ra (F /F_m \otimes \Z_{(p)} )\rtimes_{\tau} \Z $, which gives rise to an isomorphism
\begin{equation}\label{slsl} ( ([\pi_K ,\pi_K ] /[\pi_K ,\pi_K ]_m) \otimes \Z_{(p)}) \rtimes \Z \cong ( F /F_m \otimes \Z_{(p)}) \rtimes_{\tau } \Z, \end{equation}
and $ \tau$ with $m=1$ lies in the symplectic group $ \Sp(2g;\Z_{(p)})\subset \GL(2g;\Z_{(p)})$. Here, $g= \mathrm{deg}\Delta_K/2$, and we regard any automorphism $ F/F_1 \otimes \Z_{(p)} \ra F /F_1 \otimes \Z_{(p)}$ as an element of $\GL(2g;\Z_{(p)})$.

\end{thm}


\begin{defn}\label{bb3312}
Take a knot $K \subset M$. If $p$ and $ \lc \Delta_K $ are coprime, we call such an automorphism $\tau : F/F_m \otimes \Z_{(p)} \ra F /F_m \otimes \Z_{(p)}$ {\it a (nilpotently) $p$-localized monodromy (of $K$)} (Later, we discuss the choices of $\tau$; see \S \ref{Ss4}). If $p$ and $ \lc \Delta_K $ are not coprime, we define a $p$-localized monodromy of $K$ by $1$.
\end{defn}
Before turning to the proofs, we briefly review the relation between lifts of homomorphisms and the second group homology (see Proposition \ref{Projprop1} below). Let $\rho: \Gamma \to G$ be a group homomorphism. Let us consider the following group homomorphisms:
\[
{\normalsize
\xymatrix{
0\ar[r] & N \ar[rr]^{\kappa} & & \widetilde{\Gamma} \ar[rr]& & \Gamma \ar[d] ^{ \rho} \ar[r] & 1& (\textrm{central extension})\\
0\ar[r] & K\ar[rr] & & \widetilde{G} \ar[rr]^{p} & & G \ar[r] & 1& (\textrm{central extension}).
}}\]
Here we assume that the center of $\widetilde{G}$ is equal to $K$.
We now discuss the existence of a lift $\widetilde{\rho}: \widetilde{\Gamma} \to \widetilde{G}$ of $\rho$.
Recall the connecting map $\delta : H^1( N; K ) \ra H^2( \Gamma ; K)$ in the five-term exact sequence of group cohomology. Since $H^2(G; K)$ is in one-to-one correspondence with the set of equivalence classes of central extensions of $G$ (see, e.g., \cite[Sections 3--4]{BT}), we obtain the associated cohomology $2$-class.

\begin{prop}
[{See, e.g., \cite[Propositions 2.1.8 and 2.1.9]{BT}}]\label{Projprop1} In the above notation, the homomorphism $\rho$ admits a lift $\widetilde{\rho}: \widetilde{\Gamma} \to \widetilde{G}$ if and only if there exists a homomorphism $\alpha: N \to K$ such that
$ \delta ' (\alpha) = \rho^* ( \sigma_{\widetilde{G} }) $. Here, we identify $H^1(N;K) $ with $\Hom(N,K )$.

Furthermore, suppose $\kappa N \subset [\widetilde{\Gamma}, \widetilde{\Gamma}]$. Then another such lift $\rho'$ is conjugate to $\rho$ if and only if $\alpha = \alpha'$, where $\alpha'$ is the homomorphism $N \to K$ associated with $\rho'$ and satisfying $\delta ' (\alpha') = (\rho')^* ( \sigma_{\widetilde{G} }) $.
\end{prop}

\begin{proof}
[Proof of Proposition \ref{bb127}] The proof is by induction on $m$. If $m=2$, it is enough to define $\tau \otimes \Z_{(p)} $ by the localization of $\tau$ at $p$.

Let us assume $m>2 $ and a homomorphism $\tau\otimes \Z_{(p)} : G/G_m \otimes \Z_{(p)} \ra G/G_m \otimes \Z_{(p)} $ satisfying the required condition. Let $\alpha$ be the localization of the restriction of $ \tau $ at $p$:
\[\alpha= (\tau\otimes \Z_{(p)})_* : (G_m/G_{m+1})_{(p)} = \frac{(G/G_k \otimes \Z_{(p)})_m }{( G/G_k \otimes \Z_{(p)})_{m+1}} \lra \frac{(G/G_k \otimes \Z_{(p)})_m }{( G/G_k \otimes \Z_{(p)})_{m+1}} . \]
By (R2) and using diagram chasing on the five-term exact sequences, we can easily check that $\delta '(\alpha) = \rho^* ( \sigma_{\widetilde{G} }) $. Therefore, applying the setting,
\[ \rho= \tau\otimes \Z_{(p)}, \ \ \ \ G= \Gamma= G/G_m \otimes \Z_{(p)}, \ \ \ \ \widetilde{G}= \widetilde{\Gamma}= G/ G_{m+1}\otimes \Z_{(p)},\]
to Proposition \ref{Projprop1}, we have $\widetilde{\rho} : G/G_{m+1} \otimes \Z_{(p)} \ra G/G_{m+1} \otimes \Z_{(p)} $. If $\widetilde{\rho} $ is replaced by $\tau\otimes \Z_{(p)}$, it satisfies the required properties by construction. Moreover, by Proposition \ref{Projprop1}, the construction of $\widetilde{\rho}$ is unique up to conjugacy.
\end{proof}

Turning now to the proof of Theorem \ref{bb1279}, we need two lemmas. Since $H_*(M \setminus K;\Z ) \cong H_*(S^1;\Z) $ by Alexander duality, there is uniquely an epimorphism $ \pi_1(M \setminus K ) \ra \Z $ up to signs, and we can take the infinite cyclic covering space $E_K^{\infty} $.

\begin{lem}
[{cf. \cite[Assertion 5]{Mil}}]\label{kkk5} Suppose $p$ and $\lc \Delta_K $ are relatively prime. Then, the homology $H_1(E_K^{\infty} ;\Z_{(p)}) $ with coefficients $\Z_{(p)}$ is isomorphic to $ ( \Z_{(p)})^{\mathrm{deg} \Delta_K} $.
\end{lem}

\begin{proof}
Choose a genus-$g$ Seifert surface of $K$. According to \cite[Theorem 6.5]{Lic}, there is a finite presentation of $H_1(E_K^{\infty} ;\Z_{(p)}) $ as
\[ (\Z_{(p)}[t^{\pm 1} ])^{2g} \xrightarrow{\ A-tA^t \ } (\Z_{(p)}[t^{\pm 1} ])^{2g} \lra H_1(E_K^{\infty} ;\Z_{(p)}) \ra 0 \]
such that det$A-tA^t = \Delta_K $. Since the leading coefficient invertible in $\mathbb Z_{(p)}$, $\Delta_K $ may be monic in $\Z_{(p)}$. Therefore, $H_1(E_K^{\infty} ;\Z_{(p)}) $ is free and finitely generated over $\Z_{(p)} $. After tensoring with $\Q$, we can verify by using the elementary divisor theorem that the rank is determined by the degree of $\Delta_K(t) $. Hence, $H_1(E_K^{\infty} ;\Z_{(p)}) \cong ( \Z_{(p)})^{\mathrm{deg}\Delta_K} $, as desired.
\end{proof}

\begin{lem}[{Localization at $p$ of Stallings theorem \cite[Theorem 7.3]{Sta}}]\label{kkk129} Let $f: G \ra L$ be a group homomorphism such that $ H_1(f; \Z_{(p)} )$ is an isomorphism and $ H_2(f; \Z_{(p)})$ is surjective. Assume that, for any $m \in \mathbb{N}$, $ G_m/G_{m+1} \otimes \Z_{(p)}$ is finitely generated and free over $\Z_{(p)}$. Then, $f$ induces isomorphisms,
\[ (G_{k-1}/G_k) \otimes \Z_{(p)} \cong (L_{k-1}/L_k) \otimes \Z_{(p)} ,\] and induces embeddings $ G/G_k \hookrightarrow L/L_k$ for any $k \in \mathbb{N}$.
\end{lem}
\noindent
The key point here is that, under the given assumption, the first (co)homology groups in the universal coefficient theorem do not affect the second ones. Therefore, an argument similar to the proof of \cite[Theorem 7.3]{Sta} applies (we omit the proof).
\begin{proof}
[Proof of Theorem \ref{bb1279}] First, let us review the Milnor pairing \cite{Mil}. Let $E_{K}^{\infty } $ be the cyclic covering space, and let $t$ be the covering transformation on $ H^1(E_{K}^{\infty } ;\Z_{(p)})= \Z_{(p)}^{2g}$. Let $ \mathcal{G}$ denote $ [\pi_K, \pi_K]= \pi_1 (E_{K}^{\infty })$. This is known as the Milnor pairing (see \cite[Assertion 9 and Remark 1]{Mil}), and there is a non-degenerate anti-symmetric bilinear map $ \langle , \rangle : H^1(E_{K}^{\infty } ;\Z_{(p)}) \times H^1(E_{K}^{\infty } ;\Z_{(p)}) \ra \Z_{(p)}$ satisfying $ \langle t a , t b \rangle = \langle a ,b \rangle $ for any $a, b \in H^1(E_{K}^{\infty };\Z_{(p)}) $. This implies the existence of a symplectic basis on $ H^1(E_{K}^{\infty } ;\Z_{(p)})$ with respect to $ \langle , \rangle $. Namely,
this $t $ can be regarded as an element of $\Sp(2g ;\Z_{(p)})$. In addition, we claim $H_2( \mathcal{G};\Z_{(p)})=0$. To see this, consider the long exact sequence,
\[ 0 \ra H_2(E_{K}^{\infty } ;\Z_{(p)}) \ra H_2(E_{K}^{\infty }, \partial E_{K}^{\infty } ;\Z_{(p)})
\ra H_1(\partial E_{K}^{\infty } ;\Z_{(p)}) \xrightarrow{\ (\mathrm{inclusion})_* \ }H_1( E_{K}^{\infty } ;\Z_{(p)}).\]
By Milnor pairing again, the second and third terms are $\Z_{(p)}$. The last map is zero since $H_1(\partial E_{K}^{\infty };\Z_{(p)}) \cong \Z_{(p)}$ is generated by a longitude, which is bounded by a Seifert surface in $E_{K}^{\infty } $. Therefore, $H_2(E_{K}^{\infty } ;\Z_{(p)})=0 $. Since $ \pi_1 E_{K}^{\infty } = \mathcal{G}$ and there is an epimorphism $H_2(E_{K}^{\infty } ;\Z_{(p)}) \ra H_2( \pi_1(E_{K}^{\infty } );\Z_{(p)}) $, we have $H_2( \mathcal{G};\Z_{(p)})=0$, as desired.

Let us construct an isomorphism below \eqref{seq84}. Choose a homomorphism $\phi: F \ra \mathcal{G}$ whose induced map on the first homology $H_1(\bullet, \Z_{(p)})$ sends the generators of $F$ to the symplectic basis. Then, $\phi$ induces an isomorphism on $H_1(\bullet , \Z_{(p)})$, which can be regarded as an element of $\Sp(2g;\Z_{(p)})$. Since $H_2( \mathcal{G};\Z_{(p)})=0$ and $H_1( \mathcal{G};\Z_{(p)}) $ is free by Lemma \ref{kkk5}, it follows from Lemma \ref{kkk129} that $\phi$ induces $F/F_m \ra\mathcal{G}/\mathcal{G}_m$, which ensures isomorphisms $F_k/F_{k+1} \otimes \Z_{(p)} \cong (\mathcal{G}_k/\mathcal{G}_{k+1} )\otimes \Z_{(p)}$ and
\begin{equation}\label{seq84} \Psi_k: F/F_k \otimes \Z_{(p)} \stackrel{\sim}{\lra} (\mathcal{G}/\mathcal{G}_{k}) \otimes \Z_{(p)} , \end{equation}
for any $k\in \N.$ By Proposition \ref{bb127}, an action of $\Z$ on $ \mathcal{G} $ yields that of $ (\mathcal{G}/\mathcal{G}_{k}) \otimes \Z_{(p)} $, which gives an automorphism $\tau : F/F_k \otimes \Z_{(p)} \ra F/F_k \otimes \Z_{(p)}$ via $\Psi_k$. Hence, by construction, the composite
\[\psi_m: \pi_K \xrightarrow{ \rm proj. } (\mathcal{G}/\mathcal{G}_{k}) \rtimes \Z
\xrightarrow{ \rm localization. } (\mathcal{G}/\mathcal{G}_{k}\otimes \Z_{(p)} ) \rtimes \Z
\xrightarrow{ \Psi_k^{-1} \rtimes \mathrm{id}_{\Z} } (F/F_{k}\otimes \Z_{(p)} ) \rtimes_{\tau} \Z \]
has the desired properties. 
\end{proof}

\begin{rem}\label{okura}
For a concrete knot $K \subset S^3$, the above proof implies that, to compute $\tau$ explicitly, we need to describe a symplectic basis of $H^1(E_{K}^{\infty };\Z_{(p)}) $ explicitly. Fortunately, Ohkura \cite{Oh} gives a list of matrix presentations of the Milnor parings on $H^1(E_{K}^{\infty };\Z_{(p)}) $ with $p=0$, where $K$ is one of ``quasi-Pretzel knots". Thus, it is not so hard to concretely obtain $0$-localized monodromies according to this list.
\end{rem}
\begin{rem}\label{okura2}
Furthermore, when $p=0$, we can easily check that Theorem \ref{bb1279} holds for any knot $K$ in any rational homology sphere $M$. 
Since the proof is essentially identical to the preceding one,
we omit the details.
\end{rem}

\section{Knot invariants from conjugacy classes of $\Out (F/F_k \otimes \Z_{(p)}) $}
\label{Ss4}
Here, we will discuss the choices of localized monodromies of a knot (Proposition \ref{ii9}) and suggest a procedure to obtain knot invariants (Theorem \ref{ii98788}). In what follows, Aut$(G)$ denotes the automorphism group of $G$ and $\mathrm{Inn}(G) $ (resp. $\Out(G)$) is the inner (resp. outer) automorphism group of $G$. Here, $\Out(G) $ is defined to be $\Aut(G) / \mathrm{Inn}(G) $.

\begin{prop}\label{ii9}
Let $G$ be a group. Take $\varphi, \psi \in \mathrm{Aut}(G)$. We denote the classes in $\Out (G)$ by $\overline{\varphi}$ or $ \overline{\psi}$. Then, the following are equivalent:

\begin{enumerate}[(I)]
\item There is a group isomorphism $\rho : G \rtimes_{\varphi} \Z \cong G \rtimes_\psi \Z$, whose quotient on $\Z$ is the identity map.
\item The classes $ \overline{\psi}$ and $\overline{\varphi}$ are conjugate in $\Out (G)$.
\end{enumerate}
\end{prop}

\begin{proof}
Suppose (I). By assumption, there are $g$ and an automorphism $ \alpha \in \Aut(G)$ satisfying $\rho (e ,1)= (g,1)$ and 
$ \rho( x,0)= (\alpha (x),0 )$ for any $x \in G$. Notice the equality $(e,1)(x,0)(e,1)^{-1}= (\varphi (x),0) \in G \rtimes_{\varphi} \Z $. The application of $\rho$ to the left-hand side is computed as
\[ \rho ( (e,1)(x,0)(e,1)^{-1})= (g,1)(\alpha(x),0)(\psi ^{-1}(g^{-1}),-1)= (g \psi(\alpha(x)) g^{-1},0), \]
while its application to the right side implies $ \rho( (\varphi (x),0) )= (\alpha(\varphi(x)) ,0)$. Namely, $g \psi(\alpha(x)) g^{-1} =\alpha(\varphi(x))$, which means (II).

Conversely, if we assume (II), we have $g \psi(\alpha(x)) g^{-1} =\alpha(\varphi(x))$ for some $g\in G$ and $\alpha \in \mathrm{Aut}(G) $. Then, the correspondence defined by $(x,0) \mapsto (\alpha(x),0)$ and $(e,1) \mapsto (g,1)$ gives rise to an isomorphism $ G \rtimes_{\varphi} \Z \cong G \rtimes_\psi \Z $, as desired.
\end{proof}

Next, we will discuss symplecticity and the choices of localized monodromies. Let $q_k : \Aut( F/F_k \otimes \Z_{(p)}) \ra \Aut(F/F_1 \otimes \Z_{(p)}) =\GL(2g;\Z_{(p)} )$ be the epimorphism induced by the projection $ F/F_k \ra F/F_1$, $\mathrm{SpAut}( F/F_k \otimes \Z_{(p)} )$ be the subgroup $q_k^{-1}(\Sp(2g;\Z_{(p)}))$ and $\mathrm{SpOut}( F/F_k \otimes \Z_{(p)} )$ be the subgroup of $\Out(F/F_k \otimes \Z_{(p)} )$ generated by $q_k^{-1}(\Sp(2g;\Z_{(p)}))$. For a knot $K$ in $M$, the localized monodromy $ \tau $ of $K$ lies in $ \mathrm{SpAut}( F/F_k \otimes \Z_{(p)} )$ by Theorem \ref{bb1279}. As a modification of Proposition \ref{ii9}, we will classify the choice of the monodromies.

\begin{thm}\label{ii98788}
For a knot $K$ in an integral homology 3-sphere $M$, suppose that $ p$ and $\lc \Delta_K(t)$ are relatively prime. Choose a $p$-localized monodromy $ \tau \in \mathrm{SpAut}( F/F_k \otimes \Z_{(p)} )$ of $K$. Then, the conjugacy class of $\tau$ in $\mathrm{SpOut}(F/F_k \otimes \Z_{(p)} )$ depends only on the knot type of $K$.
\end{thm}

\begin{proof}
Let $\psi_k , \psi_k'$ be $p$-localized monodromies of the knot $K$. Recall the isomorphisms $\Psi_k, \Psi'_k$ from \eqref{seq84}, which are constructed from choices of the symplectic basis of $ F/F_1 \otimes \Z_{(p)}$. Since $\Sp(2g ;\Z_{(p)})$ transitively acts on the set of all symplectic bases of $ F/F_1 \otimes \Z_{(p)}$, the composite $\Psi_1^{-1} \circ \Psi_1'$ with $k=1$ can be regarded as an element of $\Sp(2g ;\Z_{(p)})$. Thus, by Proposition \ref{Projprop1}, the lift $\Psi^{-1}_k \circ \Psi_k'$ with $k>1$ can be regarded as an element of $\mathrm{SpAut}(F /F_k \otimes \Z_{(p)})$. In other words, the choice of $F /F_k \otimes \Z_{(p)}$ as a fixed group can be covered by conjugacy of $ \mathrm{SpAut}(F/F_k \otimes \Z_{(p)} )$. Recall that $H_1( \pi_1(M \setminus K) ;\Z) \cong \Z $ by Alexander duality. Hence, the proof of Proposition \ref{ii9} implies that $ \psi_k$ and $ \psi_k'$ are conjugate in $\mathrm{SpOut}(F /F_k \otimes \Z_{(p)})$.
\end{proof}

In conclusion, we can use several algorithms to get knot invariants. As one of them, we reach the following definition:
\begin{defn}\label{def115}
Take a representation $\rho : \mathrm{SpOut}(F/F_k \otimes \Z_{(p)} ) \ra \GL(V)$ for some finite-dimensional vector space $V$. For a knot $K$, we call the eigenvalue polynomial of $\rho(\tau) $ {\it the meta-nilpotent Alexander polynomial} (associated with $\rho$).
\end{defn}

\begin{exa}\label{ex115}
If $k=1$ and $V=H_1( E_K^{\infty }; \Q)$ and $\rho $ is the inclusion, the eigenvalue polynomial of $\rho(\phi) $ is the classical Alexander polynomial $\Delta_K$. In fact, this $\rho(\phi) $ can be regarded as a characteristic polynomial of $H_1([\pi_K,\pi_K];\Q)$, which is known to be $\Delta_K$; see \cite[Theorem 6.17]{Lic}.
\end{exa}

\begin{rem}\label{ex125}
Recall that the eigenvalue polynomial $\lambda(t)$ of a symplectic matrix $A$ satisfies a symmetry in the sense of $\lambda(t^{-1}) = t^{-2 \mathrm{rk} A}\lambda(t) $. In particular, if $\rho: \mathrm{SpOut}(F/F_k \otimes \Z_{(p)}) \ra \GL(V)$ is a symplectic representation, the meta-nilpotent Alexander polynomial has a symmetry as well.
\end{rem}

\begin{exa}\label{ex117}The next example is taken from \cite{Co} (see also \cite{Ka} for a Jacobi-diagrammatic construction). Given a finite-dimensional Lie algebra $\mathfrak{g}$ over $\Q$, the authors considered a linear equivalence relation $\sim$ on the universal enveloping algebra $U(\mathfrak{g})^{\otimes 2g}$ and constructed a representation $ \mathrm{Out}(F /F_k\otimes \Q ) \ra \mathrm{End}(U(\mathfrak{g})^{\otimes 2g} )/\! \sim $ \footnote{Strictly speaking, the paper \cite{Co} only defines a representation $\mathrm{Out}(F) \ra \mathrm{End}(U(\mathfrak{g})^{\otimes 2g} )/ \! \sim $; however, from the definition, if we replace $\sim$ with an appropriate equivalence relation, the construction yields representations of $\mathrm{Out}(\lim_{\infty \leftarrow k} F/F_k \otimes \Q)$.}.
It is interesting to analyze knot invariants from this perspective.
\end{exa}
\noindent
In summary, for such an invariant, it is important to concretely construct a representation $ \mathrm{SpOut}(F/F_k \otimes \Z_{(p)} ) \ra \GL(V)$ for some $p>k$. In this paper, we will give other ways to obtain knot invariants. 

\section{Comparison to Dehn-Nielsen embedding}
\label{93553}
As a summary of the results in the preceding sections, this section compares the localized meta-nilpotence of knots and fibered knots.

For this, we need Proposition \ref{jj2} below. Since $F/F_k$ is known to be torsion-free nilpotent, Lemma \ref{kkk129} gives an embedding $\iota : F/F_{k} \hookrightarrow F/F_{k}\otimes \Z_{(p)} $. Moreover, as will be shown in \eqref{seq5} later, this induces an injective homomorphism $ \Aut(F/F_{k} ) \hookrightarrow \Aut(F/F_{k} \otimes \Z_{(p)})$. Furthermore, we show the following (see Section \ref{9914} for the proof):
\begin{prop}\label{jj2} Let $p=0$. Since $\Z_{(0)} = \Q$, we will use $\Q$ rather than $\Z_{(p)}$. The quotient of the injection to the outer automorphism groups $\iota_*: \Out(F/F_{k} ) \rightarrow \Out(F/F_{k} \otimes \Q) $ is also injective.
\end{prop}
\noindent
Let us explain \eqref{seq367}. Let rk$F=2g $, and $\mathcal{M}_{g,1}$ be the mapping class group of the one-punctured surface $\Sigma_{g,1} $. Recall the Dehn-Nielsen embedding $ \mathcal{M}_{g,1} \hookrightarrow \SpOut(F) $; see, e.g., \cite[Theorem 8.8]{FM}; as mentioned in \cite[Page 446]{Mo2}, the inverse limit according to $F/F_k \ra F/F_{k-1}$ induces an injective homomorphism $\SpOut(F) \hookrightarrow \SpOut( \lim_{\infty \leftarrow k }  F/F_k )$, where the injectivity is immediately obtained from the embedding $F \hookrightarrow \lim_{\infty \leftarrow k }  F/F_k $. To summarize, by Proposition \ref{jj2}, we have the injections,
\begin{equation}\label{seq367}
\mathrm{DN}: \mathcal{M}_{g,1} \hookrightarrow \mathrm{SpOut}(\lim_{\infty \leftarrow k } (F/ F_k \otimes \Q )) \subset \mathrm{Out}(\lim_{\infty \leftarrow k } (F/ F_k \otimes \Q )).
\end{equation}

We now make the comparison in the diagram \eqref{khkh2} below.
Starting from a fibered knot $K$ with monodromy $\phi_K \in \mathcal{M}_{g,1}$, the monodromy $\tau$ obtained in Theorem \ref{bb1279} is, by construction, equal to $\mathrm{DN}(\phi_K)$.
Let $M$ be an integral homology sphere.
The above discussion can be summarized in the following commutative diagram:
\begin{equation}\label{khkh2}{
\xymatrix{ \mathcal{M}_{g,1} /\textrm{conj.} \ar@{^{(}->}[rrr]^{ \textrm{DN embedding} }
& & & \displaystyle{\bigsqcup_{ p \in \mathrm{Spec}\Z}} \mathrm{SpOut} (\lim_{\infty \leftarrow k } (F/ F_k \otimes \Q)) /\textrm{conj.}\\
\{ \textrm{oriented fibered knots of genus }g \}
\ar@{^{(}->}[u]
\ar@{^{(}->}[rrr]^{\rm inclusion} & & & \{ \textrm{knots in }M \textrm{ with deg}\Delta_K=2g\} . \ar[u]
}} \end{equation}
Here, the vertical maps represent the correspondences between a knot $K$ and its ($p$-localized) monodromy $\tau$. In particular, the right vertical map extends the classical correspondence between a fibered knot and its monodromy.

This diagram is suggestive. For instance, it is natural to ask about the injectivity and the image of the right vertical map. The left vertical map is easily seen to be injective; in particular, knot invariants derived from $\tau$ can completely classify fibered knots. However, the right vertical map is not injective. For example, if $K'$ is a knot with Alexander polynomial $\Delta_{K'} = 1$, then, by construction, the localized monodromies of $K$ and $K \sharp K'$ coincide.

It remains an open question whether the restriction of the right vertical map to the set of alternating knots, or to the set of $\Q$-homologically fibered knots, is injective. Here, a knot $K \subset M$ is said to be $\Q$-{\it homologically fibered} if twice its Seifert genus equals $\deg \Delta_K$. As is known (see, e.g., \cite{MM}), every (pseudo-)alternating knot is $\Q$-homologically fibered; moreover, all knots with fewer than $12$ crossings are $\Q$-homologically fibered except for
$11_{n34}\allowbreak,11_{n42}\allowbreak,11_{n45}\allowbreak, 11_{n67}\allowbreak,11_{n73}\allowbreak, 11_{n107}\allowbreak, 11_{n152}$; see \cite{CL}.

On the other hand, let us discuss a subgroup of $\mathrm{SpOut} (F/F_k \otimes \Z_{(p)})$, which should contain the image of the right vertical map, as an almost surjective map.
What is the minimal subgroup of $ \mathrm{SpOut} ( \lim F /F_k  \otimes \Z_{(p)}) $ that contains the image of the right vertical map?
This section gives partial answers as shown in Propositions \ref{llk1} and \ref{llk1456} below.


To state the propositions, we need some terminology. 
Let $F$ be the free group on $2g$ generators $x_1, \dots, x_{2g}$, identified with $\pi_1(\Sigma_{g,1})$, where the generator $\zeta$ of $\pi_1(\partial \Sigma_{g,1}) \cong \Z$ is regarded as $[x_1,x_2] \cdots [x_{2g-1},x_{2g}]$. 
The image of the Dehn-Nielsen embedding is the subgroup of $\mathrm{Aut}(F)$ consisting of elements $\phi$ satisfying $\phi(\zeta) = \zeta$.
Using the same notation as in \cite{GL} and \cite[Section 3]{Mo2}, let us consider the subgroup of the form
\[ \Aut_0 (F/F_k \otimes \Z_{(p)})' := \{ \phi \in \SpAut (F/F_k \otimes \Z_{(p)} ) \ | \ \phi( \iota_{k} (\zeta)) =\iota_{k} ( \zeta) \ \}, \]
where $\iota_{k} $ is the inclusion $F/F_k \hookrightarrow F/F_k \otimes \Z_{(p)}$, and define
\begin{equation}\label{seq37}
\Aut_0 (F/F_k \otimes \Z_{(p)}):= p_{k+1}(\Aut_0 (F/F_{k+1} \otimes \Z_{(p)})' ),
\end{equation}
where $p_{k+1}$ is the projection $\SpAut (F/F_{k+1} \otimes \Z_{(p)}) \ra \SpAut (F/F_{k} \otimes \Z_{(p)})$. Integrally speaking, we can define a subgroup $ \Aut_0 (F/F_k) $ of
$\Aut_0 (F/F_k)$ in a similar way.
\begin{prop}\label{llk1}
Let $K \subset M$ be a $\Q$-homologically fibered knot of genus $g$. Then, there is a localized monodromy of $K$, which lies in the subgroup $\Aut_0 (F/F_k \otimes \Z_{(p)})$.
\end{prop}
\noindent
We give the proofs in Appendix \ref {Ap2}. 

Meanwhile, we discuss a realization of localized monodromies. Precisely,
\begin{prob}\label{llk14532}
Let $\tau $ in the subgroup $ \Aut_0 (F/F_k \otimes \Q)$ satisfy that
$ \mathrm{id}_{F/F_1}- q_k ( \tau) : F/F_1 \otimes \Q \ra F/F_1 \otimes \Q$ is an isomorphism.
Then,
is there a knot in a rational homology sphere such that the monodromy is equal to $\tau$?
\end{prob}
\noindent
Here, we can easily see that the bijectivity of $\mathrm{id}_{F/F_1}- q_k ( \tau) $ is a necessary condition for the existence of a rational homology sphere.

First,
in integral cases, we obtain a simple result:
\begin{prop}\label{llk145}
Let $\tau $ in the subgroup $ \Aut_0 (F/F_k)$ satisfy that
$ \mathrm{id}_{F/F_1}- q_k ( \tau) : F/F_1 \ra F/F_1$ is an isomorphism.
Then,
there is a knot in an integral homology sphere such that the monodromy is equal to $\tau$.
\end{prop}
Next, we discuss rational cases, and give partial answers to the above problem.
\begin{prop}\label{llk1456}Let $p=0.$
\begin{enumerate}[(1)]
\item Let $k=1$, and $\tau \in \Aut_0 (F/F_k \otimes \Q) =\Sp(2g;\Q) $ satisfy that
$ \mathrm{id}_{F/F_1}- \tau : F/F_1\otimes \Q \ra F/F_1\otimes \Q$ is an isomorphism.
There is a knot in a rational homology sphere such that the monodromy is $\tau.$
\item Let $k>1. $ For any $\Q$-homologically fibered knot in a rational homology sphere with the monodromy $\tau \in \Aut_0 (F/F_k \otimes \Q) $ and any $ \eta \in \mathrm{ISpAut}_0 (F/F_k)$,
there is a knot in a rational homology sphere with the monodromy $\tau \eta \in \Aut_0 (F/F_k \otimes \Q) $.
\end{enumerate}
\end{prop}
We also give the proofs in Appendix \ref {Ap3}, with a relation to homology cobordisms.

\section{The automorphism groups $ \mathrm{Aut}( F/F_k \otimes \Z_{(p)})$ and $ \mathrm{Out}( F/F_k \otimes \Z_{(p)})$ }
\label{9914}
Before constructing knot invariants from Theorem \ref{ii98788}, we will study the automorphism groups $ \mathrm{SpAut}( F/F_k \otimes \Z_{(p)})$ and $ \mathrm{SpOut}( F/F_k \otimes \Z_{(p)})$ in detail. Throughout this section, we will employ the usual notation of the mapping class group as in \cite{DJ, ES, Mo, Mo2, Sa} and denote $F/F_1 \cong \Z^{\mathrm{rk} F}$ by $ H$, $F/F_k $ by $ N_k$, and $F_{k-1}/F_{k} $ by $ \mathcal{L}_k$. By $H \otimes \Z_{(p)}$, we mean $H_{(p)}$; By $H_{(p)}^* $, we mean $H \otimes \Z_{(p)} $ the dual space.

We will start with an observation by S. Morita \cite[\S 2]{Mo} and review the sequences \eqref{seq3} and \eqref{seq4}. Consider $\Hom ( H , \mathcal{L}_{k+1} )$ as a $\Z$-module. For a homomorphism $f : H \ra \mathcal{L}_{k+1}$, define $\tilde{f}: N_{k+1} \ra N_{k+1} $ by $\tilde{f} (\gamma )=f([ \gamma]) \gamma $, where $\gamma \in N_{k+1}$ and $[ \gamma] \in H $ is the class of $\gamma $ under the projection $N_{k+1} \ra N_{1} =H$. Then, the correspondence $ f \mapsto \tilde{f}$ gives rise to a homomorphism $\iota : \Hom ( H , \mathcal{L}_{k+1}) \ra \Aut(N_{k+1}) $. Furthermore, let $q_* : \Aut(N_{k+1}) \ra \Aut(N_{k}) $ be the map induced by the projection $ N_{k+1} \ra N_k$. According to Proposition 2.3 in \cite{Mo}, these maps can be summarized as
\begin{equation}\label{seq3}
0 \ra \Hom ( H , \mathcal{L}_{k+1}) \stackrel{\iota }{\lra } \Aut(N_{k+1})
\stackrel{q_* }{\lra } \Aut(N_{k}) \ra 1 \ \ \ \ \ (\mathrm{group \ extension}),
\end{equation}
which is not central. To see the extension is central, we define the subgroup $I(N_k)$ of $\Aut(N_k)$ to be the kernel of the projection $\Aut(N_k) \ra \Aut(N_1) =\GL(H)$. Then, the restriction of \eqref{seq3} written as
\begin{equation}\label{seq4}
0 \ra \Hom ( H , \mathcal{L}_{k+1}) \stackrel{\iota }{\lra } I(N_{k+1}) \stackrel{q_* }{\lra } I(N_k)\ra 1
\end{equation}
is known to be a central extension. Hence, $I ( N_k )$ is a nilpotent group. Here, we should note the fact that the projection $\Aut(N_k) \ra \Aut(N_1) =\GL(H)$ does not split.

In addition, we will examine the localizations of the above sequences. Let $I(N_{k}\otimes \Z_{(p)}) $ be the kernel of the projection $\Aut(N_k \otimes \Z_{(p)}) \ra \GL(H_{(p)})$. Then, the localization of \eqref{seq4} yields a central extension,
\begin{equation}\label{seq5}
0 \ra \Hom ( H_{(p)} , \mathcal{L}_{k+1} \otimes \Z_{(p)}) \xrightarrow{ \ \iota\otimes \Z_{(p)} \ } I(N_{k+1}\otimes \Z_{(p)})
\stackrel{q_* }{\lra } I(N_{k}\otimes \Z_{(p)}) \ra 1.
\end{equation}
Furthermore, according to Proposition 2.6 in \cite{Mo} (which is essentially the Levi decomposition), if $p = 0$, then the projection yields a split extension.
\begin{equation}\label{seq6}
0 \ra I(N_{k}\otimes \Z_{(p)}) \lra \Aut(N_k \otimes \Z_{(p)})
\lra \GL(H_{(p)}) \ra 1 \ \ \ \ \ (\mathrm{split \ extension}),
\end{equation}
and two such splittings are conjugate to each other (In Section \ref{99sec}, we will see that this extension splits even if $k\leq 3$ and $ p>3$). In summary, as a symplectic restriction, $\mathrm{SpAut}(N_k \otimes \Z_{(p)}) $ is also a semi-direct product of the nilpotent group $I(N_{k}\otimes \Z_{(p)}) $ and $\Sp(H_{(p)}) $. Since $ \Aut(F/F_1) \cong \Out(F/F_1) $, the outer $\mathrm{SpOut}(N_k \otimes \Z_{(p)}) $ is a semi-direct product of a nilpotent group and $\Sp(H_{(p)}) $.

Furthermore, when $p=0$, concerning the subgroup $\Aut_0(N_{k}\otimes \Q ) $ in \eqref{seq37}, it has been shown \cite[Theorem 3.3]{Mo2} that the submodule $
( \iota \otimes \Q)^{-1} (\Aut_0(N_{k}\otimes \Q ) \subset \Hom ( H_{(0)} , \mathcal{L}_k\otimes \Q) $ is equal to the kernel of the bracketing:
\begin{equation}\label{seq63}
\Qer([,]: H_{(0)} \otimes (\mathcal{L}_{k}\otimes \Q) \lra \mathcal{L}_{k+1} \otimes \Q) ,
\end{equation}
where $H_{(0)} \cong H_{(0)}^* $ by Poincar\'{e} duality, we identify $\Hom (H_{(0)} , \mathcal{L}_{k}\otimes \Q)$ with $H_{(0)} \otimes (\mathcal{L}_{k}\otimes \Q) $. Since the bracketing is $\Sp$-equivariant, the kernel \eqref{seq63} is a $\Sp$-submodule; the subgroup $\Aut_0(N_{k}\otimes \Q ) $ turns out to be also a semi-direct product of a nilpotent group and $\Sp(H_{(0)}) $.

Next, we will examine the automorphism groups $\mathrm{Inn} (N_k \otimes \Z_{(p)} ) \subset \Aut(N_k \otimes \Z_{(p)}) $. For a group $G$, let $ Z(G)$ denote the center of $G$. Notice that the inner automorphism group $\mathrm{Inn}(N_k \otimes \Z_{(p)}) $ is isomorphic to $ N_k \otimes \Z_{(p)} /Z(N_k \otimes \Z_{(p)})$, which is isomorphic to $N_{k-1} \otimes \Z_{(p)}$ because of the universality (R1)(R2) in \S \ref{Ss45} with the five lemma. Further, the inclusion $ N_{k-1} \otimes \Z_{(p)} =\mathrm{Inn}(N_k \otimes \Z_{(p)}) \subset \Aut(N_k \otimes \Z_{(p)}) $ is described as follows: For $ b \in \mathcal{L}_{ k}$, we define the homomorphism
\[\lambda_k (b) : H \lra \mathcal{L}_{k+1} ; \ \ h \longmapsto [b,h], \]
which the reader should keep in mind. Since $\GL(H) =\Aut(F/F_1) \supset \mathrm{Inn} (F/F_1)= 0 $, we have $\mathrm{Inn} (N_k \otimes \Z_{(p)} ) \subset I(N_k \otimes \Z_{(p)} ) $, which yields the following diagram:
\begin{equation}\label{seq48}
{\normalsize
\xymatrix{
0 \ar[r] & \mathcal{L}_{ k}\otimes \Z_{(p)} \ar[r] \ar@{^{(}->}[d]^{\lambda_k \otimes \Z_{(p)} } & N_{k} \otimes \Z_{(p)} \ar[r]
\ar@{^{(}->}[d]& N_{k-1} \otimes \Z_{(p)} \ar[r] \ar@{^{(}->}[d]& 1 &( \mathrm{central \ extension} )\\
 0 \ar[r] & \Hom ( H_{(p)} , \mathcal{L}_{k+1}\otimes \Z_{(p)}) \ar[r]^{ \ \ \iota }
& I(N_{k+1}\otimes \Z_{(p)}) \ar[r]& I(N_{k}\otimes \Z_{(p)}) \ar[r]& 1 & (\mathrm{central \ extension}) .
}}
\end{equation}
This diagram is commutative by the definitions of $ \iota$ and $\lambda_k$. Here, it is worth noting that $\lambda_k $ (resp. $\lambda_k \otimes \Z_{(p)} $) is equivariant under the action of $ \GL(H )$ (resp. $ \GL(H_{(p)} )$). In addition, we will show an $\Sp$-equivariance: precisely,

\begin{lem}\label{ii92280}
Suppose $\mathrm{rk}F=2g$ and $p\neq 2$. The inclusion $\lambda_k \otimes \Z_{(p)} $ is equivariant under the action of $ \Sp(H_{(p)} )$.
\end{lem}

Finally, we are now in a position to give the proofs of Proposition \ref{jj2} and Lemma \ref{ii92280}. The reader may skip these proofs on a first reading.

\begin{proof}
[Proof of Proposition \ref{jj2}] The proof proceeds by induction on $k$.
When $k = 1$, the statement follows immediately from the inclusion $\GL(H)= \Out(F/F_1) \hookrightarrow \Out(F/F_1 \otimes \Q )= \GL(H_{(0)}) $. 
Now assume that $\Out(N_k) \to \Out(N_k \otimes \Q)$ is injective for some $k $.
By \eqref{seq48}, we obtain the following commutative diagram:
\[{\normalsize
\xymatrix{ 
0 \ar[r] & \Coker ( \lambda_k ) \ar[r] \ar[d]
& \Out(N_{k+1}) \ar[r] \ar[d]& \Out(N_{k}) \ar[r] \ar@{^{(}->}[d] & 0 & (\mathrm{central \ extension})\\
0 \ar[r] & \Coker ( \lambda_k \otimes \Q ) \ar[r]
& \Out(N_{k+1}\otimes \Q) \ar[r]& \Out(N_{k}\otimes \Q) \ar[r]& 0 & (\mathrm{central \ extension}),
}}
\]
where the vertical maps are the localizations at 0. 
If we can show that the image of $\lambda_k : \mathcal{L}_{k} \lra H \otimes \mathcal{L}_{k+1}$ is a direct summand as a $\Z$-module,
the induced map $\Coker ( \lambda_k ) \ra \Coker ( \lambda_k \otimes \Q) $ is injective; 
Thus, by diagram chasing, the middle vertical map turns out to be injective, as required.

It suffices to analyze the image of $\lambda_k$ as follows.
Let $\mathfrak{p}_{2g} \subset \Hom(H, \mathfrak{L}) \oplus \Hom(\mathfrak{L}, \mathfrak{L})$ denote the tangential derivation algebra of the free Lie algebra $\mathfrak{L}$, and let $\mathfrak{p}_{2g}(k)$ be its homogeneous submodule of degree $k$.
Then, it is known (see, e.g., \cite[\S 3]{Satoh3}) that $\mathfrak{p}_{2g}(k)$ is a direct summand of $H \otimes \mathcal{L}_{k+1}$,
and that the correspondence $e_i \otimes h \mapsto [e_i, h]$ induces an $\mathfrak{S}_{2g}$-equivariant $\Z$-module isomorphism $H \otimes \mathfrak{L}_{k}  \cong \mathfrak{p}_{2g}(k)$.
Notice that the homomorphism $ \mathfrak{L}_{k} \ra H \otimes \mathfrak{L}_{k}; a \mapsto \sum_{i=1}^{2g }e_i \otimes a$ is a splitting injection; hence, the image of $\lambda_k$ is itself a direct summand, as required.
\end{proof}

\begin{proof}
[Proof of Lemma \ref{ii92280}] Let $J$ be a $(2g \times 2g)$ matrix of the form $\begin{pmatrix}0 & -1 \\1 & 0 \end{pmatrix} \oplus \cdots \oplus \begin{pmatrix}0 & -1 \\1 & 0 \end{pmatrix}.$ Suppose that $e_1, \dots, e_{2g}$ is a symplectic basis of $H_{(p)}$ with respect to the intersection form $\langle v,w\rangle= \ ^t v J w $ where $v,w \in H_{(p)}$. For any $ h \in \Sp(H_{(p)}) $, it is sufficient to show $ h \cdot (\sum_{i=1}^{2g} e_i^* \otimes [e_i, \ell ] )= \sum_{i=1}^{2g} e_i^* \otimes [e_i, h \cdot \ell ] $ for any $\ell \in \mathcal{L}_{k+1}$. Since the symplectic group $\Sp(H_{(p)}) $ is generated by transvections (see the Introduction of \cite{Ch}), we may suppose the existence of $a \in \Z_{(p)}, w \in H_{(p)}$ such that $h \cdot v = v + a \langle v,w \rangle w$ for any $v \in H_{(p)}$. Expand $w$ as $\sum_{j=1}^{2g}w_j e_j$ for some $w_j \in \Z_{(p)}$. Then, we have
\[e_{2j}^*\cdot h = (e_{2j} J h J)^* = e_{2j}^* + a \langle e_{2j-1},w \rangle w J= e_{2j}^* - a w_{2j} J w^*\]
and $(e_{2j-1} \cdot h )^* = e_{2j-1}^* - a w_{2j-1} (J w)^* $. Therefore, we obtain the desired equality as follows:
\[ h \cdot (\sum_{i=1}^{2g} e_i^* \otimes [e_i, \ell ] ) =
\sum_{i=1}^{2g} h \cdot e_i^* \otimes [h \cdot e_i, h \cdot \ell]  \]
\[ = \sum_{i=1}^{2g} (e_i^* - a w_i J w^*) \otimes [ e_i + a \langle e_i,w \rangle w , h \cdot \ell] \]
\[ = \sum_{i=1}^{2g} \bigl(e_i^* \otimes [e_i, h \cdot \ell ] - J w^* \otimes [a w_i e_i, h \cdot \ell ] \bigr)+\sum_{i=1}^{g} a \bigl(e_{2i-1}^* \otimes w_{2i}[w, h \cdot \ell] -e_{2i}^* \otimes w_{2i-1}[w, h \cdot \ell] \bigr) \]
\[
+ \sum_{i=1}^{g} a ^2 \bigl(w_{2i} J w ^* \otimes w_{2i-1}[w, h \cdot \ell] - w_{2i-1} J w ^* \otimes w_{2i}[w, h \cdot \ell] \bigr) \]
\[ = \sum_{i=1}^{2g} \bigl(e_i^* \otimes [e_i, h \cdot \ell ] \bigr) -a J w^* \otimes [w, h \cdot \ell ] +a J w^* \otimes [w, h \cdot \ell ] + 0 = \sum_{i=1}^{2g} e_i^* \otimes [e_i, h \cdot \ell ] .\]
\end{proof}

\section{Observations of $\mathrm{SpOut}(N_k \otimes \Z_{(p)})$ with $k=2,3$}
\label{99sec}
According to Theorem \ref{ii98788}, it is important to analyze quantitatively the conjugacy classes of $\mathrm{Out}(N_k \otimes \Z_{(p)})$ and $\mathrm{Aut}_0(N_k \otimes \Z_{(p)})$.

As examples with $k=2$ and $k=3$, this section will analyze $\mathrm{SpOut}(N_k \otimes \Z_{(p)}) $ in detail. We will see that the Enomoto-Satoh trace \cite{ES, Sa1} is applicable to our construction (Section \ref{ooi23}). Throughout this section, we suppose $p \neq 2$ and $p\neq 3$, and let $ \GL(H_{(p)} )$ act on $\Hom( H_{(p)}, \mathcal{L}_k \otimes \Z_{(p)})$ naturally by
\begin{equation}\label{seq8}
(Af)(u) := Af(A^{-1} u) \ \ \ \ \ \ \ (A \in \GL( H_{(p)})),
\end{equation}
for $f \in \Hom( H_{(p)}, \mathcal{L}_k \otimes \Z_{(p)})$ and $u \in H_{(p)}$.

\subsection{Example 1: $\mathrm{SpOut}(N_k \otimes \Z_{(p)}) $ with $k=2$}
\label{93883}
Before we describe $\Out( N_2 \otimes \Z_{(p)})$, we will review an observation in \cite[Page 204]{Mo}. Let $H_{(p)} \wedge H_{(p)} $ be the exterior square of $H_{(p)} $. Let $\mathcal{G}_2$ be $(H_{(p)} \wedge H_{(p)}) \times H_{(p)} $ with the group operation,
\[ ( a,v) \cdot ( b,w) := (\frac{1}{2}v\wedge w + a +b,v+w ), \ \ \ \ \ \ \ (v,w \in H_{(p)} , \ a,b \in H_{(p)} \wedge H_{(p)}).\]
Then, $N_2\otimes \Z_{(p)} \cong \mathcal{G}_2$ is known (see \cite{Mo}). Consider the action of $(f,A) \in \Hom( H_{(p)}, \Lambda^2 H_{(p)}) \rtimes \GL(H_{(p)} ) $ on $\mathcal{G}_2$ defined by
\[ (f,A)(\xi,u):= (2f(Au)+ A \xi , Au ) \ \ \ \ \ ((\xi,u) \in \mathcal{G}_2 ) . \]
\begin{prop}\label{312aa}
Let $p \neq 2$. 
The homomorphism $\Xi: \Hom( H_{(p)}, \Lambda^2 H_{(p)}) \rtimes \GL(H_{(p)} ) \ra \Aut(N_2 \otimes \Z_{(p)})$ induced from this action is an isomorphism.
 \end{prop}   
\begin{proof} 
Under the isomorphism $N_2\otimes \Z_{(p)} \cong \mathcal{G}_2 $,
the restriction of $\Xi$ coincides with $\iota \otimes \Z_{(p)}$ in \eqref{seq5}, and therefore it is an isomorphism onto $I(N_2 \otimes \Z_{(p)})$. By definition, $\Xi$ is $\GL(H_{(p)})$-equivariant, and hence its restriction to the semidirect products must also be an isomorphism.
 \end{proof}

Next, we will give decompositions of $\Hom( H_{(p)}, \Lambda^2 H_{(p)}) $, and discuss normal subgroups of $\Aut(N_2 \otimes \Z_{(p)}) $. Consider the $\Sp$-equivariant inclusion $I_3 : \Lambda^3 H_{(p)} \hookrightarrow \Hom( H_{(p)}, \Lambda^2 H_{(p)}) $ given by
\[ I_3( v \wedge w \wedge u ) (z)= \langle z , w\rangle u\wedge v + \langle z , u\rangle v\wedge w+\langle z , v\rangle w\wedge u,\]
where $v,u,w,z \in H_{(p)}$ and the symbol $\langle , \rangle $ means the intersection form. Notice that $\Lambda^3 H_{(p)}$ contains an $\Sp$-submodule of the form $ \{ h \wedge (\sum_{i=1}^g e_{2i-1}\wedge e_{2i} ) \ | \ h \in H_{(p)}\} $ if $g>1$. Furthermore, as a cokernel of $I_3$, let us consider the subspace of $H_{(p)}\otimes \Lambda^2 H_{(p)} $ generated by $ u \otimes( v\wedge w)- v\otimes ( w\wedge u)$ for $u,v,w \in H_{(p)}$. That is,
\begin{equation}\label{77711}
\mathfrak{t}= \langle \ \ u \otimes( v\wedge w)- v\otimes ( w\wedge u) \ \ | \ u,v,w \in H_{(p)} \ \rangle \subset H_{(p)}\otimes \Lambda^2 H_{(p)}.
\end{equation}
Here, $\dim (\mathfrak{t}) = (8g^3 -2g)/3 $ is known. In summary, by the symplectic Littlewood-Richardson rule (see \cite{KT}) in the context of the Young tableau, we have a decomposition of $\Sp$-representations:
\begin{equation}\label{777111}
\Hom( H_{(p)}, \Lambda^2 H_{(p)}) \cong H_{(p)} \otimes \Lambda^2 H_{(p)} \cong
\Lambda^3 H_{(p)} \oplus \mathfrak{t} \cong ( [1^3] \oplus [1]) \oplus ([2, 1] \oplus [1]),
\end{equation}
if $g>1$. Recall from Lemma \ref{ii92280} that the image $\mathrm{Im}(\lambda_1\otimes \Z_{(p)})$ is an $\Sp$-representation subspace; we have

\begin{lem}\label{lem49}
The image of $\lambda_1 \otimes \Z_{(p)}$ is contained in the subspace $\mathfrak{t}$. Namely, $\Im ( \lambda_1 \otimes \Z_{(p)}) $ is the last term $[1]$ in \eqref{777111}.
\end{lem}

\begin{proof}
Let $e_1, \dots, e_{2g} \in H_{(p)}$ be a standard basis of $H_{(p)}$. We can verify that the image of $\lambda_1$ is given by $ \{ f_b \in \Hom( H_{(p)}, \Lambda^2 H_{(p)}) | b \in H_{(p)} \}$, where $f_b$ sends $a \in H_{(p)}$ to $a \wedge b \in \Lambda^2 H_{(p)}= \mathcal{L}_2 \otimes \Z_{(p)}$. Hence, the image $\mathrm{Im}(\lambda_1\otimes \Z_{(p)})\subset H_{(p)} \otimes \Lambda^2 H_{(p)} $ has a basis $ \sum_{\ell=1}^g e_{2 \ell -1} \otimes (e_{2 \ell} \wedge e_j) + e_{2 \ell } \otimes (e_{2 \ell -1} \wedge e_j)$ with $ 1\leq j \leq n$. By applying $u=e_{2 \ell}, v =e_{2 \ell -1 },w=e_j$ to \eqref{77711}, we can readily see that the basis is contained in $ \mathfrak{t}$. Hence, we have $\mathrm{Im}(\lambda_1 \otimes \Z_{(p)}) \subset \mathfrak{t}$, as required.
\end{proof}

Notice the classical fact (see, e.g., \cite{DJ}) that the kernel $ \Qer([,])$ in \eqref{seq63} is isomorphic to $ \cong [1^3] \oplus [1] $ and contains $ \Lambda^3 H_{(p)} $ in \eqref{777111}. To summarize, we determine $\mathrm{SpOut}( N_2 \otimes \Z_{(p)}) $ as follows:
\begin{prop}\label{tuki}
If $g=1$, then $\mathrm{SpOut}( N_2 \otimes \Z_{(p)}) $ and $ \mathrm{SpOut}_0( N_2 \otimes \Z_{(p)})$ are isomorphic to $\Sp(H_{(p)})$. If $g>1,$ there are group isomorphisms
\[\mathrm{SpOut}( N_2 \otimes \Z_{(p)}) \cong ([1^3] \oplus [1]\oplus [2,1]) \rtimes \Sp(H_{(p)}), \ \ \ \ \ \mathrm{SpOut}_0( N_2 \otimes \Z_{(p)}) \cong ([1^3] \oplus [1]) \rtimes \Sp(H_{(p)}). \]
\end{prop}

\subsection{Example 2: $\SpOut( N_k \otimes \Z_{(p)})$ with $k=3$}
\label{9723}
The structure of $\Aut( N_3 \otimes \Z_{(p)})$ when $p=0$ is discussed in \cite{Mo4}. However, even if $p>4,$ we will explicitly express the group structures of $N_3 \otimes \Z_{(p)} $ and $\Aut( N_3 \otimes \Z_{(p)})$ and make an observation about $\SpOut( N_3 \otimes \Z_{(p)})$.

Let us introduce a bilinear map $\nu : H_{(p)}\times \Lambda^2 H_{(p)} \ra H_{(p)}\otimes \Lambda^2 H_{(p)}$ defined by
\begin{equation}\label{888}
\nu (u, v\wedge w) = \frac{2}{3} u\otimes ( v\wedge w) -\frac{1}{3}v\otimes ( w\wedge u) - \frac{1}{3}w\otimes ( u\wedge v).
\end{equation}
Recall the subspace $\mathfrak{t}$ of $H_{(p)} \otimes \Lambda^2 H_{(p)} $ in \eqref{77711}. For $f \in \Hom( H_{(p)}, \Lambda^2 H_{(p)} ) $, let us define a linear map $\Upsilon_{f} : \Lambda^2 H_{(p)} \ra \mathfrak{t} $ by
\begin{equation}\label{7771}
\Upsilon_{f}(u \wedge v):= \nu(u, f(v))- \nu(v, f(u ) ).
\end{equation}
Then, we can define the group structure on $\Hom( H_{(p)}, \mathfrak{t} ) \times \Hom( H_{(p)}, \Lambda^2 H_{(p)}) \times \GL(H_{(p)}) $ by
\begin{equation}\label{777}
(F',f', A') (F,f, A):= (F'+ A' F+ \Upsilon_{f'}(A'f( \bullet)) -\Upsilon_{A' f}(f' ( \bullet)) , f'+A'f, A'A).
\end{equation}

\begin{prop}\label{7787}
Let $L_3 $ be the product $\mathfrak{t} \times \Lambda^2 H_{(p)} \times H_{(p)} $. For $(\alpha, a, v), (\beta,b,w) \in L_3$, we define the binary operation $(\alpha, a, v) \cdot (\beta,b,w)$ on $L_3$ by
\begin{equation}\label{777122}
( \alpha + \beta +
\frac{1}{2}\nu( v,b) -\frac{1}{2}\nu(w,a)+ \frac{1}{12} \nu( v,v \wedge w) +\frac{1}{12} \nu( w,w \wedge v), a+b + \frac{1}{2} v \wedge w, v+w).
\end{equation}
Consider the action of the group \eqref{777} on $L_3$ defined by
\begin{equation}\label{7771222}
(F,f, A) \cdot (\alpha, a, v) :=
\bigl(2F(Av)+A \alpha + \Upsilon_{f}( 2f(Av) + 2Aa) , 2f(Av)+Aa ,Av \bigr).
\end{equation}
Then, the operation \eqref{777122} gives rise to a group structure in $L_3$, and the operation \eqref{7771222} preserves the group operation. Furthermore, the group $L_3$ is isomorphic to $N_3 \otimes\Z_{(p)} $, and the homomorphism
\[\mu : \Aut( N_3 \otimes \Z_{(p)}) \lra \Hom( H_{(p)}, \mathfrak{t} ) \times \Hom( H_{(p)}, \Lambda^2 H_{(p)}) \times \GL(H_{(p)}) \]
induced by \eqref{7771222} is a group isomorphism.
\end{prop}

\begin{proof}
Although it is a bit complicated to check the former part, the verification is elementary. So, we will leave it for the reader to verify. Instead, we will show $ N_3 \otimes \Z_{(p)} \cong L_3$. Let $f_1, \dots, f_{2g}$ be a basis of $F$, and define a homomorphism $\psi: F \ra L_3 $ by $\psi(f_i) =(0,0,e_i)$. Then, we can easily verify that $\phi$ induces $\Phi: N_3 \otimes \Z_{(p)} \ra L_3 $, since the quotient $ \Phi $ modulo $F_4 \otimes \Z_{(p)} $ gives an isomorphism $N_3 \otimes\Z_{(p)} \cong \Lambda^2 H_{(p)} \times H_{(p)}$ and the restriction of $\Phi $ on $ F_2/F_3 \otimes \Z_{(p)}$ gives an isomorphism $F_2/F_3 \otimes \Z_{(p)} \cong \mathfrak{t}$; we can easily notice that the projection $L_3 \ra \Lambda^2 H_{(p)} \times H_{(p)}$ is a central extension. Therefore, $\Phi$ is a group isomorphism by the five lemma.

Finally, we will show that $\mu $ is an isomorphism. We can easily see that the quotient of $ \mu$ modulo $ \Hom( H_{(p)}, \mathfrak{t} )$ is an isomorphism, and the restriction of $\mu $ on $\Hom( H_{(p)}, \mathfrak{t} ) \times \{0\} \times \{1\} $ is a linear isomorphism by the definitions of $ \iota $ and \eqref{777}. Hence, $\mu$ is an isomorphism by the central extension in \eqref{seq5} and the five lemma.
\end{proof}

Next, we will examine the image of $\lambda_3 \otimes \Z_{(p)} $ (Lemma \ref{lem491} below). Thanks to the Littlewood-Richardson rule again, if $g>1$ and $p \geq 7$, we have a decomposition of $\Sp$-representations,
\[ \Hom( H_{(p)}, \mathfrak{t} ) \cong H_{(p)} \otimes \mathfrak{t}
\cong ([1] \otimes [2, 1]) \oplus ( [1] \otimes [1] ) \]
\begin{equation}\label{ppo}\cong ( [2, 1, 1] \oplus [2^2] \oplus [3,1] \oplus [2]\oplus [1^2]) \oplus ( [2] \oplus [1^2] \oplus [0] ). \end{equation}
If $g=1$, then $\Hom( H_{(p)}, \mathfrak{t} ) \cong [1] \otimes [1] \cong [0] \oplus [2]. $

\begin{lem}\label{lem491}
For $s,t \leq 2g$, let $\gamma_{s,t}$ be the linear map $ H_{(p)} \ra \mathfrak{t} $ which sends $e_x$ to $\nu( e_x, e_s \wedge e_t)$. Let $R_3$ be the subspace of $\Hom( H_{(p)}, \mathfrak{t})$ spanned by $\gamma_{s,t}$ with $s<t \leq 2g$.

Then, $ R_3$ is equal to the image of $ \lambda_2$. Furthermore, $ R_3$ is isomorphic to $[1^2]\oplus [0]$ as an $\Sp$-representation and contains the image of the map
\[\Im ( \lambda_2 \otimes \Z_{(p)} ) \times \Im ( \lambda_2 \otimes \Z_{(p)} ) \ra \Hom(H_{(p)}, \mathfrak{t}) ; \ \ (f,f') \longmapsto \Upsilon_{f'}(f( \bullet))- \Upsilon_{f}(f'( \bullet)). \]

In particular, the inner automorphism group $ \mathrm{Inn}(N_3 \otimes \Z_{(p)})$ corresponds to the subset
\[ R_3 \times [1] \times \{ \mathrm{id}_{H_{(p)}} \} \subset \Hom( H_{(p)}, \mathfrak{t} ) \times \Hom( H_{(p)}, \Lambda^2 H_{(p)}) \times \GL(H_{(p)}).\]
\end{lem}

\begin{proof}
By definition, $R_3$ is isomorphic to $\Lambda^2 H_{(p)} \cong [1^2]\oplus [0]$. For $s \leq 2g$, let $f_s$ be a linear map that sends $e_i$ to $e_i \wedge e_s$, as an element of $\Im ( \lambda_2 \otimes \Z_{(p)} ) $. Then, $ \Upsilon_{f_t} (f_s(e_x)) - \Upsilon_{f_s} (f_t(e_x))= \nu ( e_x, e_s \wedge e_t) $ by direct calculation. Therefore, the linear map $\lambda_2: e_x \mapsto [e_x, [e_s,e_t]]$ is equal to $ \Upsilon_{f_t} (f_s(\bullet)) - \Upsilon_{f_s} (f_t(\bullet))$, as required.
\end{proof}

\begin{exa}
[{The case of $g=1$}]\label{ex14k} Next, we have that the representation space $[1] \otimes [2, 1 ] $ is zero, and $\dim (H_{\Q } \otimes \mathcal{L}_3 ) =4, \ \dim (\Im ( \lambda_2) ) =1$. Hence, $\SpAut( N_3 \otimes \Z_{(p)})$ subject to $ \mathrm{Inn}(N_3 \otimes \Z_{(p)} ) $ is isomorphic to $ [2] \rtimes \Sp(H_{(p)})$. Notice that the kernel $\Qer([,])$ is $[2]$ since $\dim (\mathcal{L}_5 ) =3 $. To summarize, we conclude that
\begin{equation}\label{44} \SpOut(N_3 \otimes \Z_{(p)}) \cong [2] \rtimes \Sp(H_{(p)} ), \ \ \ \ \SpOut_0(N_3 \otimes \Z_{(p)}) \cong \Sp(H_{(p)} ). \notag \end{equation}

\end{exa}
\begin{exa}
[{The case of $g>1$}]\label{ex14k2} Meanwhile, when $g>1$, the kernel in \eqref{seq63} is known to be isomorphic to $[2^2]\oplus[1^2] \oplus [0] $ (see, e.g., \cite[page 33]{Sa}).
Therefore, we conclude that $\SpOut_0( N_3 \otimes \Z_{(p)}) $ is isomorphic to the group on $\bigl(\{ [2^2 ] \oplus [1^2] \oplus [0]\} \times \{ [1^3] \oplus [1] \} \bigr) \rtimes \Sp(H_{(p)})$.
\end{exa}

\subsection{Morita trace map and Enomoto-Satoh trace map}
\label{93123}
Let $p=0$. Morita \cite{Mo3} introduced a trace map from $ \Hom( H_{(0)}, \mathcal{L}_k )$, which gives some obstructions of the Johnson homomorphism. Satoh \cite{Sa1} defines a modified trace map from $ \Hom( H_{(0)}, \mathcal{L}_k )$, which has been studied from the viewpoint of representations as a joint work with Enomoto \cite{ES}. 

We will start by giving a brief review of the Enomoto-Satoh trace. Let $ \mathcal{L} $ be $\oplus_{k \geq 0 }\mathcal{L}_k \otimes \Q $, which is a Lie algebra induced by the bracketing. For each $k \geq 1$, let $\Phi_k: H^*_{(0)} \otimes H^{\otimes (k+1)}_{(0)} \ra H^{\otimes k}_{(0)}$ be the contraction map with respect to the first component, defined by
\[ x_i^* \otimes x_{j_1} \otimes \cdots \otimes x_{j_{k+1}} \longmapsto
x_i^*(x_{j_1}) \cdot x_{j_2} \otimes \cdots \otimes x_{j_{k+1}} . \]
Let $\iota_k : \mathcal{L}_k \otimes \Q \ra H^{\otimes k}_{(0)} $ be the natural map defined by $[X,Y] \mapsto X\otimes Y -Y \otimes X$, which gives rise to a Lie homomorphism $ \mathcal{L} \ra \oplus_{k \geq 0 }H^{\otimes k}_{(0)} $. This map is injective by the Poincar\'{e}-Birkoff-Witt theorem. Let $\mathcal{C}_{2g}(k)$ be the quotient $\Q$-vector space of $ H^{\otimes k}_{(0)}$ by the action of the cyclic group $\Z/k\Z$ on the components, that is,
\[\mathcal{C}_{2g} (k):= H^{\otimes k}_{(0)}/ \langle
a_1 \otimes a_2 \otimes \cdots \otimes a_k - a_2 \otimes a_3 \otimes \cdots \otimes a_k \otimes a_1 \ | \ a_i \in H_{(0)} \rangle . \]
Then, {\it the Enomoto-Satoh trace map} is defined to be the composite map
\[ \mathrm{Tr}_{k}^{ES} : H^*_{(0)} \otimes \mathcal{L}_{k+1} \xrightarrow{\id \otimes \iota_{k+1} } H^*_{(0)} \otimes H^{\otimes (k+1)}_{(0)} \stackrel{\Phi_k }{\lra} H^{\otimes k}_{(0)}
\stackrel{\mathrm{proj.}}{- \!\!\!\! \lra} \mathcal{C}_{2g}(k).\]
According to \cite[Section 7]{Sa}, the Morita trace map \cite{Mo3} can be summarized as the composite map $P_k \circ \mathrm{Tr}_{k}^{ES} : H^*_{(0)} \otimes \mathcal{L}_{k+1}\ra S^k H_{(0)} $, where $S^k H_{(0)}$ is the $k$-fold symmetric tensor product of $H_{(0)}$ and $P_k$ is the natural projection $\mathcal{C}_n(k) \ra S^k H_{(0)}$. We should remark that the trace maps are $\GL(H_{(0)}) $-equivariant by definition.

We present a lemma and construct the group homomorphism below \eqref{kkk28}.

\begin{lem}\label{lem4933}
Let $ k\geq 2.$ The composite $ \mathrm{Tr}_{k}^{ES} \circ ( \lambda_k \otimes \id_{\Q} ) $ is zero.
\end{lem}
\noindent
This lemma is proven at the end of this subsection. To conclude, we have a $\GL(H_{(0 )}) $-equivariant Lie homomorphism:
\[\bigoplus_{j \geq 2 }\mathrm{Tr}_{j}^{ES} : \Hom( H_{(0)}, \mathcal{L}_{j+1}) \lra \bigoplus_{j \geq 2} \mathcal{C}_{2g}(j) , \]
where the lie algebra structure of the image is trivial. Then, by the same discussion as \cite[Section 3]{Mo2} on the Lie group-algebra correspondence, we have the group homomorphism
\[ \bigoplus_{j: 2 \leq j <k} \widetilde{\mathrm{Tr}}_{j}^{ES} : I(N_k \otimes \Q) \rtimes \GL(H_{(0)}) = \Aut(N_{k} \otimes \Q) \lra \Bigl( \bigoplus_{j : 2 \leq j < k} \mathcal{C}_{2g}(j) \Bigr) \rtimes \GL(H_{(0)}). \]
Since $ \widetilde{\mathrm{Tr}}_{j}^{ES} \circ ( \id_{\Q} \otimes \lambda_j ) =0$ by construction and Lemma \ref{lem4933}, the direct sum induces the group homomorphism
\begin{equation}\label{kkk28}\bigoplus_{j: 2 \leq j <k} \widetilde{\mathrm{Tr}}_{j}^{ES} :\SpOut(N_{k} \otimes \Q) \lra \Bigl( \bigoplus_{j : 2 \leq j < k} \mathcal{C}_{2g}(j) \Bigr) \rtimes \Sp(H_{(0)}).\end{equation}

\begin{exa}
[{The case of $k=3$}]\label{exko} Notice that $\mathcal{C}_{2g}(2) \cong S^2 H_{(0)} \cong [2] $ by definition. Recall from Proposition \ref{7787} the isomorphism of $ \mathcal{L}_3 \cong \mathfrak{t}$. Then, if $g>1$, the ES trace $\Hom (H_{(0)}, \mathfrak{t}) \ra [2]$ is the projection on the first component of $[2]$, where we use the decomposition \eqref{ppo}. On the other hand, if $g=1$, the ES trace is the projection $ \Hom (H_{(0)}, \mathfrak{t})\cong [2] \oplus [0] \ra [2] $.
\end{exa}

\begin{proof}
[Proof of Lemma \ref{lem4933}] Denote $\iota_1(e_j)$ by $X_j$, and $[ e_{j_s}, [ e_{j_{s+1}}, [\cdots [ e_{j_{k-1}},e_{j_k} ] \cdots ]]]$ by $Y_s$. We can easily verify that $\mathcal{L}_k \otimes \Q$ is generated by elements of the form $[ e_{j_1}, [ e_{j_2}, [\cdots [ e_{j_{k-1}},e_{j_k} ] \cdots ]]]$ by induction on $k$. Therefore, it is enough to show $ \mathrm{Tr}_{k}^{ES} ( \sum_{i=1}^{2g} e_i^* \otimes [e_i, Y_k])=0. $ We have
\begin{equation}\label{778j}\Phi_k \circ ( \id \otimes \iota_{k+1})(e_i^* \otimes [e_i, Y]_k)=
2g \iota_{k}(Y_k)+ \sum_{i=1}^{2g}\bigl( \Phi_{k-1} ( X_i^* \otimes \iota_{k}( Y_k) \bigr) \otimes X_i ,
\end{equation}
by definition. Since $\iota_{k}(Y_k)=0 \in \mathcal{C}_{2g}(k) $ is clear, we will show that the second term vanishes. We will omit writing the symbol $\otimes.$ For any $s \leq k$, notice the equality,
\[ \sum_{i=1}^{2g}\Phi_{k-s+1} ( X_i^* \otimes \iota_{s}( Y_s )) X_{j_{s-1}} \cdots X_{j_{1}}X_i = \]
\[= \iota_{s+1}(Y_{s+1})X_{j_{s-1}}X_{j_{s-2}} \cdots X_{j_1} X_{j_s} -
\sum_{i=1}^{2g}\Phi_{k-s} ( X_i^* \otimes \iota_{s+1}( Y_{s+1}) )X_{j_s}\cdots X_{j_1} X_i \in H_{(0)}^{\otimes k}.\]
Therefore, by repeating this equality, the sum in \eqref{778j} is computed as
\begin{equation}\label{PPP}(-1)^{k}[X_{j_{k-1}} ,X_{j_k}] X_{j_{k-2}}\cdots X_{j_{1}}+
\sum_{s=1}^{k-2}(-1)^s
\iota_{s+1}( Y_{s+1})X_{j_{s-1}}\cdots X_{j_1}X_{j_s } . \end{equation}
Here, by induction on $m$ with $m < k-1$, we can immediately verify the following:
\[ \sum_{s=1}^{m}(-1)^s
\iota_{s+1}( Y_{s+1})X_{j_{s-1}}\cdots X_{j_1}X_{j_s }=(-1)^{m+1} \iota_{m+1}( Y_{m+1})X_{j_{m}}\cdots X_{j_1} \in \mathcal{C}_{2g}(k). \]
Then, \eqref{PPP} is zero, since $i_{k-1}(Y_{k-1})=[ X_{j_{k-1}} ,X_{j_k} ]$. Namely, the sum in \eqref{778j} is zero, as required.
\end{proof}

\section{Some meta-nilpotent invariants of knots}
\label{Seinv}
Here, we give some knot invariants from the viewpoints of the conjugacy classes of $\mathrm{SpOut}(F/F_k \otimes \Z_{(p)})$. We say that two knots are {\it $(k,p)$-equivalent} if the associated monodromies are conjugate in $\mathrm{SpOut}(F/F_k \otimes \Z_{(p)}).$ Throughout this section, we will denote the projection $\mathrm{SpOut}(F/F_k \otimes \Z_{(p)}) \ra \mathrm{SpOut}(F/F_1 \otimes \Z_{(p)})$ by $q_k.$ Then, by \eqref{seq3}, we canonically have a surjection:
\[\{ \textrm{knot}\}/\{(k+1,p) \textrm{-equivalent} \} \lra\!\!\!\!\! \ra \{ \textrm{knot}\}/\{(k ,p) \textrm{-equivalent} \} . \]
Thus, it is natural to discuss $\mathrm{SpOut}(F/F_k \otimes \Z_{(p)}) $ with $k$ in ascending order. Below, we examine the cases of $k=1,2,3$.

\subsection{The case of $k=1$; conjugacy classes of Sp$(H_{(0)})$ }
\label{ooi21}
We begin by analyzing the case of $k=1$, i.e., $\Sp(2g; \Z_{(p)})$. In general, it is hard to completely classify the conjugacy classes of $\Sp(2g; \Z_{(p)})$. However, it is not so hard to check whether two elements in $\Sp (2g; \Z_{(p)})$ are conjugate or not; furthermore, there are class functions of $\Sp (2g; \Z_{(p)})$. For example, there is the normal form theorem which uses the symplectic conjugate; see, e.g., \cite{Gutt}. In this subsection, to obtain knot invariants with $g=1$ as a special case, we will discuss the conjugacy classes of $\Sp(2; \Z_{(p)})$ with $p=0$ in detail.

\subsubsection{The case of $g=1$ and $p=0$}
\label{ooi22}
Let $\mathcal{O}(\Q^{\times }/ (\Q^{\times } )^2)$ be the set of all subsets of the quotient multiplicative group $ \Q^{\times }/ (\Q^{\times } )^2$. According to \cite{Ca}, we will construct a map $\mathrm{sgn}: \Sp(2; \Q) /\mathrm{conj.}\ra \mathcal{O}(\Q^{\times }/ (\Q^{\times } )^2)$.

We say that $A \in \Sp(2; \Q) $ is {\it decomposable} if the eigenvalues are distinct. Choose a decomposable $ A \in \Sp(2; \Q) $. When the eigenvalues $\lambda, \lambda^{-1} $ are contained in $\Q$, we define $ \mathrm{sgn}(A)= \Q^{\times }/ (\Q^{\times } )^2 $. Otherwise, let $E$ be the quadratic field extension $\Q( \lambda)$. Let $\bar{ }: E \ra E$ be the conjugation as the generator of the Galois group $\mathrm{Gal}(E/\Q) \cong \Z/2$. Let $v= \ ^t (x,y) \in E^2$ be an eigenvector with respect to $\lambda$. Then, by setting $\Xi = \left(\begin{smallmatrix}x & \bar{x} \\ y & \bar{y}\end{smallmatrix}\right) $, we have $ A = \Xi \left(\begin{smallmatrix}\lambda & 0 \\ 0 & 1/ \lambda\end{smallmatrix}\right) \Xi^{-1} $. Then, we define
\[ \mathrm{sgn}_{E/ \Q}(A) =\mathrm{det} (\Xi )/ (\lambda -\bar{\lambda}) \in E^{\times} \]
modulo the norm $N_{E/\Q}:= \{ c \bar{c}\in  E^{\times}  | c \in  E^{\times} \} $ and define
\[\mathrm{sgn}(A) := \text{the preimage in } \Q^{\times}/(\Q^{\times})^ 2 \text{ of }\mathrm{sgn}_{E/ \Q}(A). \]
Then, we can easily show $\mathrm{sgn}(A) = \mathrm{sgn}(gAg^{-1}) $ for any $g \in \Sp(2; \Q).$ It is known that the restriction of the pair of the trace map $\mathrm{Tr}$ and $\mathrm{sgn} $ on the sets of decomposable matrices is injective (see \cite[Page 8]{Ca}).
\begin{exa}
Ohkura \cite{Oh} gives a list of monodromies in $\Sp(2; \Q)$ of some knots, including 2-bridge knots of genus 1. For example, if $K=7_4$ and $K'=9_2$, the monodromies are $A= \left(\begin{smallmatrix}1 & 1/2 \\-1/2 & 3/4\end{smallmatrix}\right)$ and $A'= \left(\begin{smallmatrix}1 & 1/4 \\-1 & 3/4\end{smallmatrix}\right) $. We can easily check that they are not conjugate using sgn.
\end{exa}
Similarly, we can compute the signatures and distinguish some knots (e.g., $9_5$ and $13_2$). In our experience, the map sgn seems to be a strong invariant for knots of genus one.

\subsection{The case of $k=2$; Class function from $\SpOut( N_2 \otimes \Z_{(p)})$}
\label{ooi23}
Next, we address the case of $k=2$. Recall $ \SpOut(F/F_2 \otimes \Z_{(p)}) \cong \Sp(2; \Z_{(p)})$ if $g=1$; thus, we may suppose $g>1$. To obtain a class function of $ \SpOut(F/F_2 \otimes \Z_{(p)})$, we need the following lemma:

\begin{lem}\label{p3k}
Assume that $\SpOut(N_2\otimes \Z_{(p)})$ is identified with a semidirect product
$V\rtimes \Sp(H_{(p)})$ as in Proposition~\ref{312aa}.
Fix $g\in \Sp(H_{(p)})$ and let $
C(g):=\{\,h\in \Sp(H_{(p)})\mid hg=gh\,\}
$ be the centralizer of $g$.
Then for any $a,b\in V$ and any $h\in C(g)$ we have
\begin{equation}\label{3334} [(b,h)(a,g)(b,h)^{-1}=\bigl(h\cdot a+(1-hgh^{-1})b,\ hgh^{-1}\bigr) \in V \rtimes \Sp(2g; \Z_{(p)}).\end{equation} 
In particular, the image of $a$ in the coinvariants
\[
\bigl(V/(1-g)V\bigr)_{C(g)}
\]
depends only on the conjugacy class of $(a,g)$ among elements whose second component equals $g$.
\end{lem}
\begin{proof}We use the semidirect product convention
\[
(a,g)(a',g')=(a+g\cdot a',\, gg') ,\qquad (a,g)^{-1}=(-g^{-1}\cdot a,\, g^{-1}).
\]
A direct calculation gives the general conjugation formula \eqref{3334}. 
If $h\in C(g)$, then $hgh^{-1}=g$, which the right hand side turns out to 
$\bigl(h\cdot a+(1-g)b,\ g\bigr).$
Passing to the quotient $V/(1-g)V$ kills the term $(1-g)b$, and passing further to coinvariants
kills the remaining action of $C(g)$. This proves the claim.
\end{proof}
\begin{rem}
For a general $h\in \Sp(H_{(p)})$, conjugation sends $(a,g)$ to an element whose second component is
$hgh^{-1}$. Accordingly, the quotient $V/(1-g)V$ is naturally identified with $V/(1-hgh^{-1})V$
via $v\mapsto h\cdot v$. Under this identification, the class constructed above is invariant.
\end{rem}
\begin{exa}\label{p3k3}
Recall from Proposition \ref{tuki} that $\SpOut(F/F_2 \otimes \Z_{(p)}) \cong ([1^3] \oplus [1] \oplus [2,1]) \rtimes \Sp(2g; \Z_{(p)}) $. Using these isomorphisms to Lemma \ref{p3k}, we will compute some knot invariants.

For example, let us consider the knots $ K=8_{20}$ and $K'= 12n_{582}$, which are fibered of genus 2. The Alexander polynomial is $(1-t+t^2)^2 $. Then, we can verify that the monodromies $\tau$ and $\tau'$ in $\Sp(4;\Z)$ are conjugate, and that, if $p=0$ and $V= ([1^3] \oplus [1] \oplus [2,1])$, then $ V_{C(\tau)} \cong \Q^{} $. By carefully checking the elements $\Psi_2(\tau)$ and $\Psi_2(\tau' )$ in $V_{C(\tau)} $, it is not so hard to show that the monodromies $\tau$ and $\tau'$ in $V \rtimes \Sp(4;\Z)$ are not conjugate.
\end{exa}

\begin{rem}\label{iiiq}
If $p=0$ and zeros of the Alexander polynomial are distinct and not any root of unity, then $V_{C(\tau)} $ is isomorphic to zero in many cases, where $V= ([1^3] \oplus [1] \oplus [2,1])$ as above. For example, by considering the set
\begin{equation}\label{ppq}\mathcal{C}_{2g}:= \left\{ \ A \in \Sp( 2g;\Q) \ \Bigl|
\begin{array}{c}
\ \mathrm{The \ eigenvalues \ } \lambda_1,\dots, \lambda_{2g} \in \mathbb{C} \mathrm{ \ satisfy \ } \\ 
\ \lambda_{i}^{n_i} \neq \lambda_{j}^{n_j} \ \mathrm{ \ for \ any \ } i\neq j, 1 \leq n_i \leq 3 ,1 \leq n_j \leq 3
\end{array}
\right\}, \end{equation}
which is closed under conjugacy, we can see that, if $\tau \in p_2^{-1}( \mathcal{C}_{2g})$, then $V_{C(\tau)} \cong 0$. For example, if the crossing number of a knot $K \subset S^3$ is less than $12$, $K$ is hyperbolic, and $\Delta_K$ is irreducible, then $\tau \in \mathcal{C}_{2g}$. For knots with such Alexander polynomials, in order to obtain nontrivial knot invariants from $\SpOut(F/F_2 \otimes \Z_{(p)})$, we must assume $p > 0$.
\end{rem}

Finally, let us briefly compare the invariants with an extension of the first Johnson homomorphism \cite{Mo}. Let $ \mathcal{M}(2) \subset \mathcal{M}_{g,1} $ be the normal subgroup generated by all the Dehn twists along separating simple closed curves. Since $\mathcal{M}_{g,1} $ acts on $\pi_1( \Sigma_{g,1})=F$, we have a group homomorphism $\rho_k: \mathcal{M}_{g,1} \ra \Aut( F )\xrightarrow{\rm proj.}\Aut( F/F_k )$. Let $k=2$. Morita showed \cite[Theorem 4.8]{Mo} that $\rho_2$ induces an embedding $\mathcal{M}_{g,1} / \mathcal{M}(2) \hookrightarrow \frac{1}{2}\Lambda^3 H \rtimes \GL(H)$; furthermore, the restriction on the Torelli group is the first Johnson homomorphism and the image is of finite index. Thus, the difference between the invariants of fibered knots and non-fibered knots might be detected from finite information. For this reason, to obtain useful information from the invariants of a non-fibered knot, we shall suppose $p>0$.

\subsection{The case of $k=3$; Symplectic quadratic forms from $\SpOut ( N_3 \otimes \Z_{(p)})$}
\label{ooi25}
Using the description of $\SpOut ( N_3 \otimes \Z_{(p)})$ in Section \ref{9723}, we will construct knot invariants of symplectic quadratic forms. Although we can describe the invariants for every knot, the description will be quite complicated; for simplicity, we will focus on the set $\mathcal{C}_{2g} $ in \eqref{ppq} and suppose $p=0$, and construct a class function
\begin{equation}\label{ppq22}\rho: q_3^{-1}(\mathcal{C}_{2g} ) \ra \{ \textrm{isomorphism classes of symplectic quadratic forms}\}^{4g^2-3g }.
\end{equation}Let $\mathcal{D}$ be the diagonal subgroup of $\mathrm{Sp}(2g; \mathbb{C})$, which is a maximal commutative subgroup and is isomorphic to $(\mathbb{C}^\times)^g$.
For a representation $V$ of $\mathcal{D}$, we denote by $V^{\mathcal{D}}$ the $\mathcal{D}$-invariant subspace.
Then we have the following:

\begin{lem}
Recall the subspaces $ \mathfrak{t} \subset H_{(0)}\otimes \Lambda^2 H_{(0)}$ and $R_3 \cong [1^2]\oplus [0]$ defined in Lemma \ref{lem491}. The dimensions of the invariant parts $ \Hom(H_{\mathbb{C}}, \mathfrak{t}\otimes \mathbb{C} )^{\mathcal{D}}$ and of $([1^2]\oplus [0])^{\mathcal{D}}$ are $4g^2 - 2g$ and $g$, respectively. In particular, $\bigl( \Hom(H_{\mathbb{C}}, \mathfrak{t}\otimes \mathbb{C} ) /R_3\bigr)^{\mathcal{D}} $ is  of dimension $4g^2 - 3g $.
\end{lem}
\begin{proof} For the latter part, we have $([1^2]\oplus [0])^{\mathcal{D}} =(\Lambda H_{\mathbb{C}})^{\mathcal{D}}$,
which is the $g$-dimensional space spanned by $e_{2i-1} \wedge e_{2i}$ for $i \leq g$. On the other hand, $\mathfrak{t}\otimes \mathbb{C}$ is known to be spanned by the generating set
\[ \{ e_{i} \otimes (e_j \wedge e_k) - e_{j} \otimes (e_k \wedge e_i) \mid i<j, k \neq i, k \neq j \}
\cup \{ e_{i} \otimes (e_j \wedge e_i)  \mid  i \neq j \}  \]
Thus, the $\mathcal{D}$-invariant subspace of $\Hom(H_{\mathbb{C}}, \mathfrak{t}\otimes \mathbb{C} )^{\mathcal{D}}$ is of dimension $2g(2g-2) + 2g = 4g^2 - 2g$, as required. \end{proof}

For $A \in \mathcal{C}_{2g} \subset \Sp(2g;\Q), $ let $P_A \in \Sp( 2g;\mathbb{C})$ be a matrix such that $ P_A A P_A^{-1}$ is the diagonal matrix $(\lambda_1) \oplus (\lambda_2)\oplus\cdots \oplus (\lambda_{2g})$, and let $\pi: \Hom(H_{\mathbb{C}}, \mathfrak{t}\otimes \mathbb{C} ) \ra \Hom(H_{\mathbb{C}}, \mathfrak{t}\otimes \mathbb{C} /R_3)^{\mathcal{D}}$ be the projection from the direct summand. As mentioned in Remark \ref{iiiq}, 
\[1-A: \Hom( H_{(p)}, \Lambda^2 H_{(p)} ) \lra \Hom( H_{(p)}, \Lambda^2 H_{(p)} )\]
is an isomorphism. Then, we define a map, $\rho^{\rm pre}$, from the preimage $q_3^{-1}(\mathcal{C}_{2g} ) \subset \mathrm{SpOut}( N_3 \otimes \Z_{(p)} )$ by
\[\rho^{\rm pre}: q_3^{-1}(\mathcal{C}_{2g} ) \lra \bigl( \Hom(H_{\mathbb{C}}, \mathfrak{t}\otimes \mathbb{C} ) /R_3\bigr)^{\mathcal{D}}; \]
\[ (F,f,A) \longmapsto \pi \bigl(P_A \bigl(F(\bullet) -\Upsilon_f((1-A)^{-1} f(\bullet))+ \Upsilon_{(1-A)^{-1}f}(f (\bullet)) \bigr) \bigr). \]
It is not difficult to show that $\rho^{\rm pre} $ does not depend on the choice of $ P_A$, but $\rho^{\rm pre} $ is not invariant with respect to conjugacy.

To solve the non-invariance, consider the action of $\Sp(2g;\Q)$ on the sets of symplectic basis, $\mathrm{SB}_g$. Since this action is free and transitive, for $B \in \Sp(2g;\Q)$, we can uniquely choose the corresponding symplectic basis $\mu(B) $. Then, $ \rho^{\rm pre}(F,f, A)$ can be regarded as a map from $\mu(A)$ to $\bigl( \Hom(H_{\mathbb{C}}, \mathfrak{t}\otimes \mathbb{C} ) /R_3\bigr)^{\mathcal{D}} $; further, we can see that this map is a quadratic map by Proposition \ref{7787}. Therefore, we have
\[ \rho: q_3^{-1}(\mathcal{C}_{2g} ) \lra \mathrm{Map}( \mathrm{SB}_g , \bigl( \Hom(H_{\mathbb{C}}, \mathfrak{t}\otimes \mathbb{C} ) /R_3\bigr)^{\mathcal{D}}) ; (F,f,A) \longmapsto \bigl( \mu(A) \mapsto \rho^{\rm pre}(F,f,A) \ \bigr) . \]
Accordingly, it is not hard to verify that the symplectic congruence classes of such quadratic forms are invariant with respect to conjugacy. To summarize
\begin{prop}
The map \eqref{ppq22} is a class function of $\SpOut( F/F_3 \otimes \Z_{(p)})$.
\end{prop}
Although the construction of $\rho$ is a bit complicated, it is not hard to make a computer program to describe the map $\rho$, when $g=2,3$. In fact, we can check that the symplectic quadratic forms are not trivial; since symplectic quadratic forms over $\mathbb{C}$ are classified (see, e.g., \cite{BHSS}), we obtain non-trivial quantitative information from the resulting quadratic forms. However, the resulting computations are hard to describe; we will omit the details.

\appendix

\section{Proof of Proposition \ref{llk1}} 
\label{Ap2}
Here, we give a proof of Proposition \ref{llk1}. The proof uses the terminology developed in Sections \ref{Ss3}--\ref{93553}. We will often regard the complement $M \setminus K $ to be a 3-manifold obtained from an integral homology 3-sphere $M $ by removing an open tubular neighborhood of $K$.

As preparation, let us examine some fundamental groups.
Choose a Seifert surface $S$ of genus $g$, and fix a meridian $\mathfrak{m} \in \pi_1(M \setminus K)$.
Take a bicollar $S \times [-1,1]$ of $S$ such that $S \times {0} = S$.
Let $\iota_{\pm}: S \to M \setminus S$ be the embeddings whose images are $S \times {\pm 1}$. Take generating sets $W:=\{ u_1, \dots, u_{2g} \}$ of $\pi_1 S $ and $ X:=\{ x_1, \dots, x_{2g}, x_{2g+1}, \dots, x_{2g+k} \} $ of $ \pi_1(M \setminus S)$ for some $k$. Here, according to the discussion in \cite[Page 480]{Tr}, we can choose $X$ such that the linking number of the cycles represented by $u_i$ and $x_j$ is the Kronecker delta $\delta_{ij}$, while the remaining generators $x_{2g+1}, \dots, x_{2g+k} $ lie in $[\pi_1(M\setminus S),\pi_1(M \setminus S) ]$, and that $[ x_1, x_2] \cdots [x_{2g-1} , x_{2g}] $ represents a longitude in $\pi_1(M \setminus K) $. Set $y_i:=(\iota_+)_* (u_i)$ and $z_i=(\iota_-)_* (u_i)$; a van Kampen argument then yields a presentation of $ \pi_1(M \setminus K)$:
\begin{equation}\label{oo456} \langle \ \mathfrak{m} , x_1, \dots, x_{2g}, x_{2g+1}, \dots, x_{2g+k}\ | \ \mathfrak{m} y_i \mathfrak{m}^{-1}\ z_i^{-1} \ \ \ (1 \leq i \leq 2g) , r_{2g+1},\dots,r_{2g+t} \ \ \rangle, \end{equation}
for some relators $ r_{2g+1},\dots,r_{2g+t}$, which do not contain $ \mathfrak{m}$ and are contained in the commutator subgroup $[ \pi_1(M \setminus S), \pi_1(M \setminus S) ] $. Let $p: E_K^{\infty} \ra M \setminus K $ be the $\infty$-cyclic covering. Then, by using the Reidemeister-Schreier method, $ \pi_1(E_K^{\infty})= [\pi_K,\pi_K]$ is presented by
\begin{equation}\label{lklk} \langle \ x_1^{(n)}, \dots, x_{2g+k}^{(n)} \ \ \ ( n\in \Z
) \ | \ y_i^{(k)} (z_i^{(k+1)})^{-1} \ \ \ (1 \leq i \leq 2g, \ n \in \Z), r_{2g+1},\dots,r_{2g+t} \ \ \rangle, \end{equation}
and the injection $ p_*: \pi_1(E_K^{\infty}) \ra \pi_1(M \setminus K)$ is represented by the correspondence $ x_i^{(k)} \mapsto \mathfrak{m}^{-k} x_i \mathfrak{m}^k. $

\begin{proof}
[Proof of Proposition \ref{llk1}] We will use the above notation. Suppose that $(p, \lc \Delta_K )=1.$ Consider a lift, $\tilde{S}$, of the Seifert surface $S$ in $E_{K}^{\infty}$. Since $\mathfrak{l} = [x_1^{(0)}, x_2^{(0)}] \cdots [x_{2g-1}^{(0)} , x_{2g}^{(0)}] $ can be regarded as a boundary loop in $\partial \tilde{S} $ or a longitude of $K$, this $\mathfrak{l}$ commutes with the monodromy on $\pi_1( E_{K}^{\infty} )$. Define $F$ to be the free group with a basis $\{ x_{2i-1}^{(0)}, x_{2i}^{(0)} \}_{i=1}^g $. Then, the monodromy induces $\tau: F / F_k \otimes \Z_{(p)} \ra F / F_k \otimes \Z_{(p)}$, which preserves $ \iota_k( \mathfrak{l} )$.

The remaining part of the proof is to show that $\tau$ lies in $\mathrm{SpAut}(F / F_k \otimes \Z_{(p)} )$ for any $k \in \N$. Notice that, by Lemma \ref{pp5} below, the inclusion $i: \tilde{S} \hookrightarrow E_{K}^{\infty}$ induces an isomorphism
\[ i^* :H^*(E_{K}^{\infty}, \partial E_{K}^{\infty} ;\Z_{(p)} ) \lra H^*(\tilde{S}, \partial \tilde{S} ;\Z_{(p)} ). \]
Therefore, $\{ x_{2i-1}^{(0)}, x_{2i}^{(0)} \}_{i=1}^g $ represents a symplectic basis of $H^1(E_{K}^{\infty};\Z_{(p)} )$. Hence, by following the proof of Theorem \ref{bb1279}, we can show that $\tau$ lies in $\mathrm{SpAut}(F / F_k \otimes \Z_{(p)} )$, as required.
\end{proof}

\begin{lem}[{folklore, see, e.g., \cite{GK}}]]\label{pp5}
Assume that $K$ is $\Q$-homologically fibered of genus $g$, i.e.\ $2g=\deg\Delta_K$.
Moreover assume $(p,\lc(\Delta_K))=1$ (so that Lemma~\ref{kkk5} applies).
Then the inclusion $i:\tilde S\hookrightarrow E_K^\infty$ induces an isomorphism$
i_*:H_1(\tilde S;\Z_{(p)}) \xrightarrow{\cong} H_1(E_K^\infty;\Z_{(p)})$ over $ \Z_{(p)}$.
\end{lem}


\begin{proof}
It is enough to show the surjectivity of $i_*$, since
$H_1(E_K^\infty;\Z_{(p)})\cong \Z_{(p)}^{2g}$ by $2g=\deg\Delta_K$ and Lemma~\ref{kkk5}.

From \eqref{lklk}, since $x_{2g+1},\dots,x_{2g+k}$ lie in
$[\pi_1(M\setminus S),\pi_1(M\setminus S)]$, the group
$H_1(E_K^\infty;\Z_{(p)})$ admits a presentation with generators
$x_1^{(n)},\dots,x_{2g}^{(n)}$ ($n\in\Z$) and relations coming from
$y_i^{(n)}(z_i^{(n+1)})^{-1}$ ($1\le i\le 2g$, $n\in\Z$).

Let $\vec{x}^{(n)}:=([x_1^{(n)}],\dots,[x_{2g}^{(n)}])\in \Z_{(p)}^{2g}$.
Then the relations imply that there exist matrices $Y,Z\in\mathrm{Mat}(2g\times 2g;\Z)$ such that
$Y\vec{x}^{(n+1)}=Z\vec{x}^{(n)}$ for all $n$.

By the standard relationship between a Seifert matrix and the Alexander polynomial,
one has $\Delta_K(t)\doteq \det(Y-tZ)$ (up to multiplication by $\pm t^m$).
Since $\deg\Delta_K=2g$ and $(p,\lc(\Delta_K))=1$, the matrices $Y$ and $Z$
are invertible over $\Z_{(p)}$. Hence
$\vec{x}^{(n+1)}=Y^{-1}Z\,\vec{x}^{(n)}$, and thus $\vec{x}^{(n)}=(Y^{-1}Z)^n\vec{x}^{(0)}$.

Therefore $H_1(E_K^\infty;\Z_{(p)})$ is generated by $[x_1^{(0)}],\dots,[x_{2g}^{(0)}]$.
On the other hand, the image of $i_*$ contains $[y_1^{(0)}],\dots,[y_{2g}^{(0)}]$,
and in the above presentation we have $( [y_1^{(0)}],\dots,[y_{2g}^{(0)}])^T = Y\vec{x}^{(0)}$.
Since $Y$ is invertible over $\Z_{(p)}$, the generators $\vec{x}^{(0)}$ lie in $\mathrm{Im}(i_*)$.
This proves the surjectivity of $i_*$, hence $i_*$ is an isomorphism.
\end{proof}

\section{Proofs of Propositions \ref{llk145} and \ref{llk1456}}
\label{Ap3}

We give the proofs of Propositions \ref{llk145} and \ref{llk1456} from the viewpoints of homology cobordisms.

We review homology cobordisms \cite{GL}. 
A {\it homology cobordism} $(N,i^+, i^-)$ over $\Sigma_{g,1}$ consists of a compact
oriented 3-manifold $N$ with two embeddings $i^+, i^- : \Sigma_{g,1} \hookrightarrow \partial N$
such that:
\begin{enumerate}[(i)]
\item $i^+$ is orientation-preserving and $i^-$ is orientation-reversing,
\item $\partial M= i^+ (\Sigma_{g,1} ) \cup i^- (\Sigma_{g,1} ) $ and
$i^+ (\Sigma_{g,1} ) \cap i^- (\Sigma_{g,1} )=i^+ (\partial \Sigma_{g,1} ) = i^- (\partial \Sigma_{g,1} ) $,
\item $i^+ |_{\partial \Sigma_{g,1} } = i^- |_{\partial \Sigma_{g,1} }$, and
\item $i^+, i^- :H_*(\Sigma_{g,1};\Z) \ra H_*(M;\Z)$ are isomorphisms.
\end{enumerate}
Similarly, we can define {\it a rational homology cobordism} by replacing
(iv) with the condition that (iv') $i^+, i^-:H_*(\Sigma_{g,1};\Q) \ra H_*(M;\Q)$ are
isomorphisms. The maps $i^+, i^- $ are called {\it markings}.
The {\it closure}, $\mathrm{cl}(N,i^{\pm} )$, of $(N,i^{\pm} ) $ is defined as
a 3-manifold obtained by identifying $i^+(\Sigma_{g,1} ) $ and $i^- (\Sigma_{g,1} ) $.

Two homology cobordisms $(N,i^+, i^-)$ and $(N', j^+, j^-)$ over $ \Sigma_{g,1} $ are said to be
{\it isomorphic} if there exists an orientation-preserving diffeomorphism $f : N \cong N'$
satisfying $j^+ = f \circ i^+$ and $j^- = f\circ i^-$. We denote by $\mathcal{C}_{g,1}$ the set of all isomorphism classes of homology cobordisms over $\Sigma_{g,1} $. By using markings, we can endow $\mathcal{C}_{g,1}$ with
a monoid structure whose product is given by
\[(N,i^+, i^-) \cdot (N', j^+, j^-) = ( N \cup_{i^- \circ (j^+)^{-1}} N', i^+, j^- ).\]
Similarly, we can define the monoid, $\mathcal{C}_{g,1}^{\Q}$, of all isomorphism classes of rational homology cobordisms over $ \Sigma_{g,1} $, and an injection $ \mathcal{C}_{g,1} \hookrightarrow \mathcal{C}_{g,1}^{\Q} $.

According to \cite{GL}, we construct for every $k$ a homomorphism $\sigma_k: \mathcal{C}_{g,1} \ra \Aut_0(F/F_k)$.
Given $(N, i^+, i^- ) \in \mathcal{C}_{g,1}$ consider the homomorphisms $i^{\pm} _*
: F \ra \pi_1 (N)$, where the base-point is
taken in $\partial ( i^+ (\Sigma_{g,1})) = \partial ( i^-(\Sigma_{g,1}))$. Since $i^{\pm}$ are homology isomorphisms, Stallings theorem implies
that they induce isomorphisms $i^{\pm} _k: F/F_k \ra \pi_1(N)/\pi_1(N)_k$.
We then define $\sigma_k (N, i^{\pm}) = (i_k^- )^{-1} \circ ( i^+_k)$,
and can easily see that $\sigma_k (N, i^{\pm}) \in \Aut_0(F/F_k) $. 
The surjectivity of $\sigma_k$ is known (see \cite[Theorem 3]{GL}).
In a parallel way, we can rationally define a homomorphism $\sigma_k^{\Q}: \mathcal{C}_{g,1}^{\Q} \ra \Aut_0(F/F_k \otimes \Q)$.
Moreover, we can show the following by the construction of $\sigma_k.$
\begin{lem}\label{ss38} Given $(N, i^+, i^- ) \in \mathcal{C}_{g,1}$,
we suppose that $ \sigma_1 (N, i^{\pm}) - \mathrm{id}_{F/F_1}$ is an isomorphism, and regard $K$ as the boundary $\partial \Sigma_{g,1}$. 
Then, the closure $ \mathrm{cl}(N,i^{\pm} )$ can be regarded as a knot $K$ in an integral homology sphere, and the localized monodromy of $K$ is equal to $ \sigma_k (N, i^{\pm}) $.

Similarly, if $(N, i^+, i^- ) \in \mathcal{C}_{g,1}^\Q$ satisfies that
$ \sigma_1^{\Q} (N, i^{\pm}) - \mathrm{id}_{F/F_1 \otimes \Q}$ is an isomorphism,
then the closure $ \mathrm{cl}(N,i^{\pm} )$ can be regarded as a knot $K$ in a rational homology sphere, and the localized monodromy of $K$ is equal to $ \sigma_k (N, i^{\pm}) $.
\end{lem}
\begin{proof}[Proof of Proposition \ref{llk145}]
Let $\tau \in \Aut_0 (F/F_k)$ satisfy that
$ \mathrm{id}_{F/F_1}- q_k ( \tau) : F/F_1\ra F/F_1$ is an isomorphism.
Thanks to the surjectivity of $\sigma_k$, we have a homology cobordism $(N, i^{\pm})$ such that
$ \sigma_k (N, i^{\pm}) = \tau$. Hence, the closure $M \setminus K = \mathrm{cl}(N,i^{\pm} )$
satisfies the required conditions.
\end{proof}
Next, we turn to prove Proposition \ref{llk1456}.
For this, we need a lemma
\begin{lem}\label{lem:transvection_cobordism}
Let $\gamma$ be a simple closed curve in $\Sigma_{g,1}$ and let $s/t\in\Q$ with $\gcd(s,t)=1$ and $t\neq 0$.
Then there exists a rational homology cobordism $C_{\gamma,g,s/t}\in\mathcal{C}_{g,1}^{\Q}$ such that
$
\sigma_1^{\Q}(C_{\gamma,g,s/t})=[T_{\gamma}]_{s,t}.
$
\end{lem}
\begin{proof}
If $s=0$, then $[T_{\gamma}]_{s,t}=\mathrm{id}$ and we can take $C_{\gamma,g,0}:=\Sigma_{g,1}\times[0,1]$.
Hence we may assume $s\neq 0$.
Put $U:=\Sigma_{g,1}\times[0,1]$ and $K:=\gamma\times\{1/2\}\subset \mathrm{Int}(U)$.
Let $\mu,\lambda\in H_1(\partial N(K);\Z)$ be the meridian and the longitude coming from the surface framing.
Let $C_{\gamma,g,s/t}$ be the $3$--manifold obtained from $U$ by $(-t/s)$--Dehn surgery along $K$,
i.e.\ by attaching a solid torus so that the curve $t\mu-s\lambda$ bounds a disk.
Denote by $i^\pm:\Sigma_{g,1}\to \partial C_{\gamma,g,s/t}$ the inclusions of $\Sigma_{g,1}\times\{0\}$ and $\Sigma_{g,1}\times\{1\}$.
A Mayer--Vietoris argument shows that $i^\pm_*:H_*(\Sigma_{g,1};\Q)\to H_*(C_{\gamma,g,s/t};\Q)$ are isomorphisms, hence
$(C_{\gamma,g,s/t},i^+,i^-)\in \mathcal{C}_{g,1}^{\Q}$.

For $x\in H_1(\Sigma_{g,1};\Q)$, take a representative $\delta\subset \Sigma_{g,1}$ transverse to $\gamma$.
In $U\setminus \mathrm{Int}N(K)$, the cylinder $\delta\times[0,1]$ intersects the meridian $\mu$ algebraically $\langle x,[\gamma]\rangle$ times, and we obtain
\[
(i^+_*)(x)=(i^-_*)(x)+\langle x,[\gamma]\rangle\,\mu\in H_1(C_{\gamma,g,s/t};\Q).
\]
Since the surgery relation is $t\mu=s\lambda$ and $\lambda=(i^-_*)([\gamma])$, we have $\mu=(s/t)(i^-_*)([\gamma])$ in $H_1(C_{\gamma,g,s/t};\Q)$.
Therefore, we get tge required equality:
\[
\sigma_1^{\Q}(C_{\gamma,g,s/t})(x)=(i^-_*)^{-1}\circ(i^+_*)(x)
=x+(s/t)\langle x,[\gamma]\rangle[\gamma]=[T_{\gamma}]_{s,t}(x).
\]

\end{proof}

\begin{rem}
Nozaki \cite{Nozaki} constructs (in genus one) homologically fibered knots in lens spaces.
In particular, the complement of a Seifert surface gives rise to a rational homology cobordism over $\Sigma_{1,1}$; see \cite[Figure~2.1, Lemma~2.4]{Nozaki}.
We cite Nozaki's work as an illustration; the above proof gives a direct surgery construction realizing the prescribed transvection.
\end{rem}

\begin{proof}[Proof of Proposition \ref{llk1456}]
(1) Fix $\tau \in \Sp(2g;\Q) $ in assumption.
Since the symplectic group $\Sp(2g;\Q) $ is generated by transvections (see the introduction of \cite{Ch}), by Lemma \ref{lem:transvection_cobordism}, we can choose simple closed curves $\gamma_1, \dots, \gamma_n$ and integers $s_1,\dots, s_n , t_1,\dots, t_n$ such that $ q_k(\tau) =[T_{\gamma_1} ]_{s_1,t_1} \cdots [T_{\gamma_n} ]_{s_n,t_n}.$
We now define a homology cobordism $N_{\tau} $ by the composite
$ N_{\tau}:= C_{ \gamma_1 ,g,s_1/t_1} \cdots C_{ \gamma_n,g,s_n/t_n} $; the closure gives a knot of a rational homology sphere.

(2) Let $ K \subset M$ be a $\Q$-homologically fibered knot with a Seifert surface $S$ of genus $g$.
Then, the complement $M \setminus S$ can be regarded as a rational homology cobordism.
For $ \eta \in \mathrm{ISpAut}_0(F/F_k)$, the above surjectivity of $\sigma_k$ ensures a rational homology cobordism $C_{\eta}$ satisfying $\sigma_k(C_{\eta})= \eta $.
Therefore, the composite $(M \setminus S ) \cdot C_{\eta}$ takes a preimage satisfying the required condition.
\end{proof}

\subsection*{Acknowledgments}
The author sincerely expresses his gratitude to Tetsuya Ito, Yuta Nozaki, Takuya Sakasai and Takao Satoh for giving him valuable comments. 
He is sincerely thankful to the referee for the thorough reading of the manuscript and for providing many insightful comments and helpful suggestions.

\vskip 1pc

\normalsize

\noindent
Department of Mathematics, Tokyo Institute of Technology
2-12-1
Ookayama, Meguro-ku Tokyo 152-8551 Japan

\end{document}